\documentclass[12pt, twoside]{article}
\usepackage{amsmath,amsthm,amssymb}
\usepackage{times}
\usepackage{enumerate}

\usepackage{amssymb, amscd, amsmath, amsthm,  graphicx, latexsym}
\usepackage{amsfonts}
\usepackage{color}
\usepackage{multirow}

 \usepackage{hyperref}
\usepackage{lineno}

\input xy
\xyoption{all}

\def\cs{\mathop{\#}}

\def\Z{{\mathbb Z}}

\def\s{\mathfrak s}

\def\titlerunning#1{\gdef\titrun{#1}}
\makeatletter
\def\author#1{\gdef\autrun{\def\and{\unskip, }#1}\gdef\@author{#1}}
\def\address#1{{\def\and{\\\hspace*{18pt}}\renewcommand{\thefootnote}{}%
\footnote {#1}}%
\markboth{\autrun}{\titrun}}
\makeatother
\def\email#1{e-mail: #1}
\def\subjclass#1{{\renewcommand{\thefootnote}{}%
\footnote{\emph{Mathematics Subject Classification (2010):} #1}}}
\def\keywords#1{\par\medskip
\noindent\textbf{Keywords.} #1}

\newtheorem{thm}{Theorem}[section]
\newtheorem{cor}[thm]{Corollary}

\newtheorem{question}[thm]{Question}

\theoremstyle{definition}

\newtheorem{rem}[thm]{Remark}
\newtheorem{exa}[thm]{Example}
\newtheorem*{conjecture}{Conjecture}

\newtheorem*{xrem}{Remark}

\numberwithin{equation}{section}

\begin{document}


\baselineskip=17pt


\titlerunning{Periodic Knots and Correction Terms}

\title{Periodic Knots and Heegaard Floer Correction Terms}

\author{Stanislav Jabuka
\and Swatee Naik}

\date{\today}

\maketitle

\address{Jabuka: Department of Mathematics and Statistics, University of Nevada, Reno, NV 89557; \email{jabuka@unr.edu}
\and
Naik: Department of Mathematics and Statistics, University of Nevada, Reno, NV 89557; \email{naik@unr.edu}}



\begin{abstract}
We derive new obstructions to periodicity of classical knots by employing the Heegaard Floer correction terms of the finite cyclic branched covers of the knots. Applying our results to two fold covers, we demonstrate through numerous examples that our obstructions are successful where many existing periodicity obstructions fail.

A combination of previously known periodicity obstructions and the results presented here, leads to a nearly complete (with the exception of a single knot) classification of alternating, periodic, 12-crossing knots with odd prime periods. For the case of alternating knots with 13, 14 and 15 crossings, we give a complete listing of all periodic knots with odd prime periods $q>3$. 
\end{abstract}
\section{Introduction}
\subsection{Background}
The study of periodic knots is an extension of the usual framework of knot theory to the equivariant case.  Focusing on finite cyclic group actions, we say that a knot $K$ in $S^3$ is {\em periodic} if there exists an integer $q > 1$ and an orientation preserving diffeomorphism $f:S^3 \to S^3$ such that $f(K) = K$, the order of $f$ is $q$, and the fixed point set of $f$ is a circle disjoint from $K$. Any  such $q$ is called a {\em period} of $K$. The set $B= {\rm Fix}(f)$ is called the {\em axis of $f$}. We shall refer to a knot of period $q$ as {\em $q$-periodic}. The positive resolution of the Smith Conjecture \cite{MB} ensures that the axis $B$ is an unknot. 

\begin{xrem}
It is crucial to assume that $f:S^3\to S^3$ be smooth to ensure that $B$ is an unknot. Montgomery and Zippin \cite{MontgomeryZippin} constructed a homeomorphism of $S^3$ of order two, with a complicated  fixed point set.  
\end{xrem}

If $\langle f\rangle$ is the subgroup of $Diff^+(S^3)$ generated by $f$, then the orbit space $S^3/\langle f\rangle$ is diffeomorphic to $S^3$ \cite{Moise}, and we let $\pi:S^3\to S^3/\langle f\rangle$ be the associated quotient map. The knot $\overline K = \pi(K)$ is called the {\em quotient knot of $K$} and $\overline B = \pi(B)$ the {\em quotient axis} of $K$. See Figure \ref{trefoil} for an illustration of $3$-periodicity of the trefoil.  
\begin{figure}[htb!] 
\centering
\includegraphics[width=10cm]{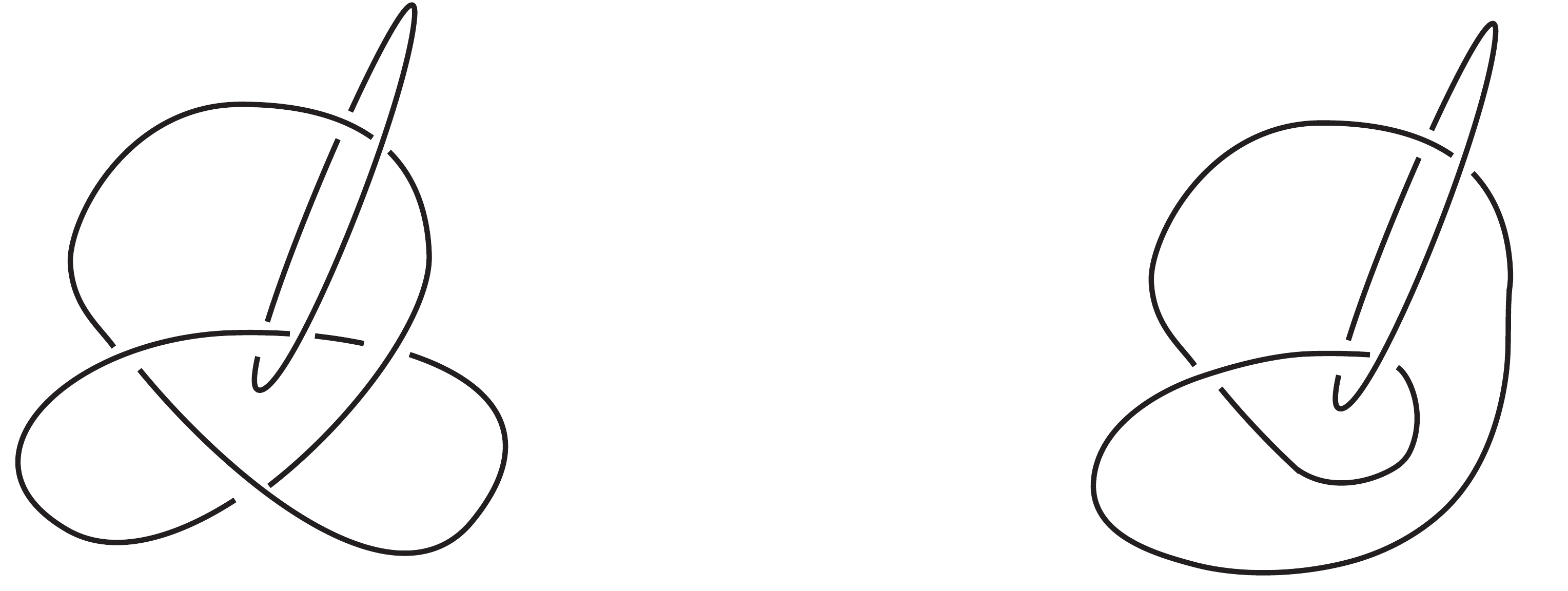}
\put(-190,20){$K$}
\put(-205,100){$B$}
\put(-8,20){$\overline K$}
\put(-10,100){$\overline B$}
\caption{A $3$-periodic diagram of the trefoil $K$ and its quotient knot $\overline K$. The axis $B$ and the quotient axis $\overline B$ are indicated.}\label{trefoil}
\end{figure}
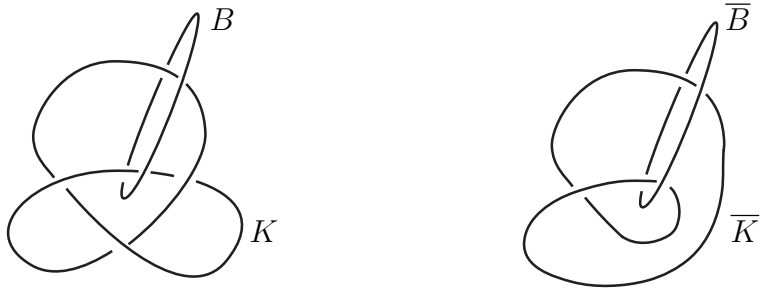
\begin{exa} \label{PeriodsOfTorusKnots}
The periods of the $(a,b)$-torus knot are the divisors of $|a|$ and $|b|$, \cite{Connor}.
\end{exa}

Topologists have studied how periodicity of a knot is reflected in its various invariants, including polynomial invariants such as the Alexander polynomial~\cite{D--L,M1,M2}, the Jones polynomial~\cite{M3} and its 2-variable generalizations~\cite{Nafaa1,Nafaa2,Prz,T2,T3,Y1,Y2}, and the twisted Alexander polynomials~\cite{HLN}.  Obstructions to periodicity have also been found in terms of hyperbolic structures on knot complements~\cite{AHW}, homology groups of branched cyclic covers~\cite{N1,N3},  concordance invariants of Casson and Gordon~\cite{N3}, Khovanov homology~\cite{Nafaa3} and link Floer homology~\cite{LinkFloer}.  

We add to these a new knot periodicity obstruction that is based on the Heegaard Floer correction terms of a cyclic branched cover of the knot. To do this, we let $Y$ be a prime-power fold cyclic cover of $S^3$ branched along $K$. Any such $Y$ is a rational homology $3$-sphere, and when $K$ is $q$-periodic, features an order $q$, orientation preserving, self-diffeomorphism $F:Y\to Y$ \cite{N1}. 
The latter is used in conjunction with the diffeomorphism invariance of the Heegaard Floer correction terms to obstruct periodicity. 
\begin{thm} \label{HeegaardFloerTheoremOnPullbackOfCorrectionTerms}
Let $Y$ be a rational homology $3$-sphere, $\s \in Spin^c(Y)$ a spin$^c$-structure on $Y$,  and $F:Y\to Y$ an orientation preserving diffeomorphism. Then 
$$d(Y,F^*(\s)) = d(Y,\s),$$
where $F^*(\s)$ is the pullback of $\s$ under $F$, and where $d(Y,\s)$ is the Heegaard Floer correction term of  $Y$ associated to the spin$^c$-structure $\s$. 
\end{thm}
Theorem \ref{HeegaardFloerTheoremOnPullbackOfCorrectionTerms} is a consequence of the diffeomorphism invariance of the Heegaard Floer groups $HF^\circ (Y,\s)$, $\circ \in \{\infty, -, +, \widehat{\phantom{HF}}\}$. Indeed, this invariance induces commutative diagrams relating the long exact sequences of the triples $\left( HF^-(Y,\s), HF^\infty(Y,\s), HF^+(Y,\s)\right)$ and  $\left( HF^-(Y,F^*(\s)), HF^\infty(Y,F^*(\s)), HF^+(Y,F^*(\s))\right)$, from which Theorem \ref{HeegaardFloerTheoremOnPullbackOfCorrectionTerms} readily follows.  
\vskip2mm
Returning to a $q$-periodic knot $K$ and its prime-power fold cyclic branched cover $Y$, understanding the fixed point set of $F^*$ allows for a prediction of certain correction terms for $Y$ to appear with multiplicities divisible by $q$, and the absence of such multiplicities obstructs $q$-periodicity of $K$. For instance, if $F^*$ is fixed-point free, then all but one value of the correction terms, appear with multiplicity divisible by $q$. 

Information about the fixed-point set of $F^*$ can be obtained by considering the first homology of $\overline Y$,  the corresponding cyclic cover of $S^3$ branched along the quotient knot $\overline K$ of $K$. If for an Abelian group $H$ and a prime $\ell$ we let $H_\ell$ denote the $\ell$-primary subgroup of $H$ (the set of elements of order a power of $\ell$), we have the following result from \cite[Proposition 2.5]{N1} the proof of which is based on a transfer argument similar to that in Section \ref{ProofOfProposition19} in the proof of
Theorem \ref{AbsenceOfL}.  
\begin{rem} \label{qln} Throughout this paper $q$ and $\ell$ will denote two distinct primes and $n$ will be a power of some prime number.
\end{rem}  
\begin{thm}[Proposition 2.5 in \cite{N1}]  \label{SwateeTheoremAboutFixedPointSets}
Let $q$, $\ell$ and $n$ be as in Remark \ref{qln}, and let $Y$ and $\overline Y$ be the $n$-fold cyclic covers of $S^3$ branched along a $q$-periodic knot $K$ and its quotient knot $\overline K$ respectively.  
Then 
\begin{equation} \label{FixedPointSetOfFStar}
\text{Fix}\left( F_*|_{H_1(Y;\mathbb Z)_\ell}\right)  \cong H_1 ( \overline {Y};\mathbb Z)_\ell.
\end{equation}
\end{thm}
\begin{rem} \label{FCompatibility}
In what follows we shall rely on an affine identification of $Spin^c(Y)$ and $H_1(Y;\mathbb Z)$. In the presence of a diffeomorphism $F:Y\to Y$, we shall further require that these identifications be \lq\lq {\em $F$-compatible}\rq\rq \, in the sense  that if $\s\in Spin^c(Y)$ and $s\in H_1(Y;\mathbb Z)$ are identified with one another, then so are $F^*(\s)$ and $F_*(s)$. 

Such an identification exists if and only if there is a spin$^c$-structure $\s_0\in Spin^c(Y)$ with $F^*(\s_0)=\s_0$, in which case the $F$-compatible affine identification is chosen so as to identify $\s_0$ with $0\in H_1(Y;\mathbb Z)$. We force the existence of such an $\s_0$ by requiring $H_1(Y;\mathbb Z)_2=0$, the latter condition implying that $Y$ possesses a unique spin-structure $\s_0$, and clearly then $F^*(\s_0) = \s_0$. The condition $H_1(Y;\mathbb Z)_2=0$ is automatic if $n$ is a power of $2$, as it will be in all of our examples (where $n=2$). 
\end{rem}
\subsection{Results and Examples}
An immediate consequence of Theorems \ref{HeegaardFloerTheoremOnPullbackOfCorrectionTerms} and \ref{SwateeTheoremAboutFixedPointSets} is our first obstruction to $q$-periodicity. 
\begin{thm} \label{ObstructionToQPeriodicityWhenFixedPointSetIsTrivial}
Let $q$, $\ell$ and $n$ be as in Remark \ref{qln}, and let  $Y$ and $\overline Y$ be the $n$-fold cyclic covers of $S^3$ branched along a $q$-periodic knot $K$ and its quotient knot $\overline K$ respectively.   Assume that $H_1(\overline Y;\mathbb Z)_\ell$ is trivial and let $F:Y\to Y$ be the diffeomorphism induced by the $q$-periodicity of $K$. Then, under an $F$-compatible affine identification of $Spin^c(Y)$ with $H_1(Y;\mathbb Z)$ (Remark \ref{FCompatibility}), the Heegaard Floer correction terms $d(Y,\s)$, corresponding to Spin$^c$-structures $\s \in H_1(Y;\mathbb Z)_\ell - \{0\}$, occur with multiplicities divisible by $q$. 
\end{thm}
\begin{exa} \label{ExampleOf12a100}
Consider the knot $K=12a_{100}$ from the knot tables \cite{knotinfo}. 
For its $2$-fold branched cover $Y$, we have $H_1(Y;\mathbb Z)_5 \cong \mathbb Z_5\oplus \mathbb Z_5$ (throughout the text we reserve the symbol $\mathbb Z_m$ for the cyclic group $\mathbb Z/m\mathbb Z$ of $m$ elements).

If this $K$ were $3$-periodic, the Alexander polynomial of any quotient knot $\overline K$ would have to be trivial, and hence $H_1(\overline Y;\mathbb Z)=0$ (a claim we justify in Section \ref{ExamplesOfPeriodicAndNonPeriodicKnots}).  According to Theorem \ref{ObstructionToQPeriodicityWhenFixedPointSetIsTrivial}, $3$-periodicity of $K$ would force the correction terms $d(Y,\s)$, $\s \in H_1(Y;\mathbb Z)_5-\{0\}$, to occur with multiplicities divisible by $3$. This does not happen as an explicit computation of the correction terms shows:
\vskip2mm
\begin{center}
\begin{tabular}{|c|c||r|r|r|r|r|} \hline
\multirow{2}{*}{$12a_{100}$}&$d(Y,\s)$ & $-\frac{4}{5}$ &  $-\frac{2}{5}$ & $0$ & $\frac{2}{5}$ & $\frac{4}{5}$ \cr \cline{2-7} 
&Multiplicity of $d(Y,\s)$ & $2$ & $6$ & $6$ & $6$ & $4$ \cr \hline
\end{tabular}\, , $\quad \s \in H_1(Y;\mathbb Z)_5 - \{0\}$.
\end{center}
\vskip2mm
It follows that $12a_{100}$ cannot be $3$-periodic.  \end{exa}

When $H_1(\overline Y;\mathbb Z)_\ell$ is nontrivial, that is when the fixed-point set of $F_*|_{H_1(Y;\mathbb Z)_\ell}$ is nontrivial, it is harder to keep tally of the multiplicities of the correction terms $d(Y,\s)$ with $\s \in H_1(Y;\mathbb Z)_\ell$, and our obstruction to $q$-periodicity in this case is weaker, taking the form of an inequality as seen in Theorem \ref{FirstEquivarianceTheorem} below.  (Note, however, that Theorem \ref{ObstructionToQPeriodicityWhenFixedPointSetIsTrivial} follows as a corollary of  Theorem
\ref{FirstEquivarianceTheorem} when $H_1(\overline Y;\mathbb Z)_\ell = 0$.)
\begin{thm} \label{FirstEquivarianceTheorem}
Let $q$, $\ell$ and $n$ be as in Remark \ref{qln}, let $K\subset S^3$ be a $q$-periodic knot and let $Y$ and $\overline Y$ be the $n$-fold cyclic covers of $S^3$ branched along $K$ and its quotient knot $\overline K$ respectively. Let $F:Y\to Y$ be the diffeomorphism induced by the $q$-periodicity of $K$. Then, under an $F$-compatible affine identification of $Spin^c(Y)$ with $H_1(Y;\mathbb Z)$ (Remark \ref{FCompatibility}), there exists a subgroup $H$ of  $H_1(Y;\mathbb Z)_\ell$, isomorphic to $H_1(\overline Y;\mathbb Z)_\ell$,  such that
\begin{itemize}
\item[(i)] Each correction term $d(Y,\s)$ with $\s \in H_1(Y;\mathbb Z)_\ell-H$, occurs with a multiplicity that is divisible by $q$.
\item[(ii)]  Let the multiplicities of correction terms $d(Y,\s)$ with $\s \in H_1(Y;\mathbb Z)_\ell$, be $n_1,n_2, \dots , n_k$ and let $m_i$ be their reductions modulo $q$, that is $m_i \equiv n_i\ ({\rm mod\ }q),$ and $0\le m_i<q$. Then 
$$m_1\, + \, m_2\, + \, \cdots \, +\, m_k \, \le \vert\, H\, \vert.$$
\end{itemize}
\end{thm}
\begin{exa}  \label{AlgPer}
Consider the knots $7_4$ and $9_2$ from the knot tables \cite{knotinfo}. For each of these, the Alexander polynomial is  
$$\Delta(t)=\Delta_{7_4}(t)=\Delta _{9_2}(t)= 4t^2 -7t + 4.$$
Note that $\Delta(-1) = 15$ showing that the first homology of the double branched cover along any knot with this polynomial is ${\mathbb Z_5}\oplus {\mathbb Z_3}$.

Let $K\, =\, 7_4 \cs 7_4 \cs 9_{2}$ and suppose that $K$ has period 3.  Then the Alexander polynomial of the quotient knot $\overline K$ is forced to be $\Delta_{\overline K}(t) = 4t^2 -7t + 4$ (as explained in Section \ref{ExamplesOfPeriodicAndNonPeriodicKnots}). Let $Y$ and $\overline Y$ denote the double branched covers along $K$ and $\overline K$, respectively. Then $H_1(Y;\mathbb Z)_5 \cong \Z _5 \oplus \Z _5 \oplus \Z _5$ and $H_1(\overline Y ; \mathbb Z)_5 \cong \Z _5.$ By Theorem \ref{FirstEquivarianceTheorem} 
the fixed point set of the generator of the $\Z _3$ action on $H_1(Y)$ is a subgroup $H$  isomorphic to $ \Z _5$ and the sum of the mod 3 multiplicities of the correction terms should be bounded above by $\vert\, H\, \vert =5$. 

The table below shows the correction terms $d(Y,\s)$ with $s\in H_1(Y;\mathbb Z)$ with their corresponding multiplicities. We see $2$ corrections terms with multiplicities $24$ and $6$, respectively,  but $9$ distinct correction terms which do not have multiplicities divisible by $3$.  By Theorem \ref{FirstEquivarianceTheorem}, $K$ cannot have period 3. Specifically, adding the modulo 3 reductions of the multiplicities gives: $2+2+2+2+1+2+1+1+1 = 10 > 5 =\vert H \vert.$ 
\vskip5mm
\begin{center}
\begin{tabular}{|c|c|c|c|c|c|c|c|c|c|c|c|c|} \hline 
$d(Y,\s)$&  $ -\frac{29}{10}$ & $-\frac{5}{2}$ & $-\frac{17}{10}$ & $-\frac{13}{10}$ & $-\frac{9}{10}$ & $-\frac{1}{2}$& $-\frac{1}{10}$& $\frac{3}{10}$& $\frac{7}{10}$& $\frac{11}{10}$&$\frac{3}{2}$\cr \hline
Multiplicity of $d(Y,\s)$ &  $8$ & $8$ & $20$ & $24$ & $8$  & $16$ & $20$ & $10$ & $6$ & $4$ &$1$\cr \hline
\end{tabular}
\end{center}
\vskip2mm
\end{exa}
In Section \ref{ExamplesOfPeriodicAndNonPeriodicKnots} we demonstrate that Examples \ref{ExampleOf12a100} and \ref{AlgPer} pass several previously known obstructions to knot periodicity, indicating that Theorems \ref{ObstructionToQPeriodicityWhenFixedPointSetIsTrivial} and \ref{FirstEquivarianceTheorem} pick out information about the knots not contained in said obstructions.  
\begin{rem} \label{RemarkAboutTwoPeriodicity}
The Heegaard Floer correction terms $d(Y,\s)$ of a rational homology $3$-sphere, are invariant under conjugation of Spin$^c$-structures: $d(Y,\s) = d(Y,\bar \s)$. This \lq\lq built-in\rq\rq \, $\mathbb Z_2$--symmetry of the correction terms,  unfortunately makes it difficult to use Theorems \ref{ObstructionToQPeriodicityWhenFixedPointSetIsTrivial} and \ref{FirstEquivarianceTheorem} to obstruct $2$-periodicity of knots. However, see \cite{LinkFloer}. 
\end{rem}

The use of correction terms to obstruct $q$-periodicity (with $q$ a prime) in Theorems \ref{ObstructionToQPeriodicityWhenFixedPointSetIsTrivial} and \ref{FirstEquivarianceTheorem}, relies on the fact that the hypotheses in said theorems assure that $F_*:H_1(Y;\mathbb Z) \to H_1(Y;\mathbb Z)$ has fixed point set $Fix(F_*)$ smaller than $H_1(Y;\mathbb Z)$. Since $q$ is prime, the cardinality of the set $\{s, F_*(s),\dots,F_*^{q-1}(s)\}$ is $q$ for each choice $s\in H_1(Y;\mathbb Z)-Fix(F_*)$, thus obtaining a $q$-fold multiple for the value of the correction terms $d(Y,s)$. 

This reasoning is completely general and guarantees the existence of $q$-fold values of correction terms $d(Y,s)$, $s\in H_1(Y;\mathbb Z)-Fix(F_*)$ whenever one has an order $q$  diffeomorphism $F:Y\to Y$. In the absence of a prime $\ell$ distinct from $q$ satisfying the hypotheses from Theorems \ref{ObstructionToQPeriodicityWhenFixedPointSetIsTrivial} and \ref{FirstEquivarianceTheorem} (as happens, for instance, when $H_1(Y;\mathbb Z)$ is itself $q$-primary), ensuring that $H_1(Y;\mathbb Z)-Fix(F_*)$ is nonempty becomes harder. Nevertheless, we submit the following nontriviality criterion. 
%
%
\begin{thm} \label{AbsenceOfL} Let $q$ be an odd prime and $n$ a power of a prime. Let $K$ be a $q$-periodic knot with quotient knot $\overline K$, and let $Y$ and $\overline Y$ be their $n$-fold cyclic branched covers. Let $F_{*,q}$ be the restriction of $F_*:H_1(Y;\mathbb Z)\to H_1(Y;\mathbb Z)$ to $H_1(Y;\mathbb Z)_q$, and let   
$$Fix(F_{*,q})\cong \mathbb Z_q^{m_1}\oplus \mathbb Z_{q^2}^{m_2} \oplus \dots \oplus Z_{q^k}^{m_k},$$
for some natural numbers $k, m_1,\dots, m_k$. Then 
\begin{equation} \label{rankInequalityForQPrimaryGroups}
|Fix(F_{*,q})| \le q^{m_1+\dots+m_k} \cdot |H_1(\overline Y;\mathbb Z)_q|.
\end{equation}
In particular, if $q^{m_1+\dots+m_k} \cdot |H_1(\overline Y;\mathbb Z)_q|< |H_1(Y;\mathbb Z)_q|$ then $H_1(Y;\mathbb Z)_q-Fix(F_{*,q})$ is nontrivial and each correction term $d(Y,s)$ with $s\in H_1(Y;\mathbb Z)_q-Fix(F_{*,q})$ occurs with a multiplicity that is divisible by $q$. 
\end{thm}
\begin{cor} \label{CorollaryForTheCaseOfQPrimaryGroups}
Under the hypothesis of the previous theorem, assume additionally that $H_1(Y;\mathbb Z)_q \cong \mathbb Z_{q^k}$ for some natural number $k>1$. Then $|Fix(F_{*,q})|\le q\cdot |H_1(\overline Y;\mathbb Z)_q|$ and if 
\begin{equation} \label{CORrankInequalityForQPrimaryGroups}
q\cdot |H_1(\overline Y;\mathbb Z)_q| <  |H_1( Y;\mathbb Z)_q|,
\end{equation}
then 
there is a set $S$ of correction terms of $Y$, such that:
\begin{itemize}
\item[(i)] Each correction term $d(Y,\s) \in S$  occurs with a multiplicity that is divisible by $q$, and
\item[(ii)] The number of correction terms in $S$, counting multiplicity, is greater than or equal to 
$ |H_1( Y;\mathbb Z)_q| -  q\cdot |H_1(\overline Y;\mathbb Z)_q|$.
\end{itemize}
\end{cor}

Examples of applications of Theorem \ref{AbsenceOfL} and Corollary \ref{CorollaryForTheCaseOfQPrimaryGroups} are given in Section \ref{SectionWithDetailsOnExamples}. 
\subsection{Applications to low-crossing knots} \label{SubSectionOnLowCrossingKnots}
Periodicity of knots with fewer than 12 crossings  has been discussed extensively in the references we have included in the introduction. For the benefit of the reader we recount some of these results below, and add some new findings to the list. As discussed in greater detail in Section \ref{SectionWithDetailsOnExamples}, the results we state here are based on subjecting knots to what we label the sieve of  \lq\lq homological periodicity obstructions\rq\rq \, first, and then running the remaining  knots through the sieve of the Heegaard Floer correction terms periodicity obstruction. As correction terms of two-fold cyclic branched covers of non-alternating knots are difficult to calculate at present, our results focus almost exclusively on alternating knots. 

By the \lq\lq homological periodicity obstructions \rq\rq \, we refer to the three periodicity tests: 
\begin{itemize} 
\item Edmonds' Genus Condition (Section \ref{EdmondsGenusCondition}), 
\item Murasugi's Alexander Polynomial Conditions (Section \ref{MurasugiSection}), and 
\item The Homology Condition (Section \ref{SectionHomologyCondition}). 
\end{itemize}
The selection of these three periodicity obstructions for comparison with the correction terms obstruction is motivated by their relatively easy implementation into Mathematica computer code. This is relevant to the examples below where we examine entire families of knots for periodicity properties, with some families containing thousands of knots. A knot by knot approach, as would for instance be required when using the twisted Alexander polynomial obstruction, would indeed by an overwhelming task.   

We are not aware of an example of a knot that passes all hitherto known periodicity obstructions but fails to pass the correction terms periodicity obstruction, though such an example may well exist. We are however able to exhibit examples of knots whose periodicity is excluded by the correction terms obstruction, even when they pass each of the three homological periodicity obstructions,  providing substantiation for the usefulness of our techniques. 
Details of the calculations presented here are provided in Section \ref{SectionWithDetailsOnExamples}.

\subsubsection{Knots with up to 9 crossings} \label{SectionOn9CrossingKnots}

Periods for knots up to 9 crossings are all known, and are listed in \cite{BZ}. The eight 3-periodic knots in this family are 
$${3_1}, \, {8_{19}}, \, {9_1}, \, {9_{35}}, \, {9_{40}}, \, {9_{41}}, {9_{47}},\, {9_{49}}.$$ 
Indeed, $3_1$, $8_{19}$ and $9_1$ are the $(3,2)$-, $(4,3)$- and $(9,2)$-torus knots respectively (see Example \ref{PeriodsOfTorusKnots}),  while $9_{35}$ is the 3-stranded pretzel knot $P(-3,-3,-3)$.  Periodic diagrams for $9_{40}$, $9_{41}$, $9_{47}$ and $9_{49}$ can be found on page 276 of \cite{BZ}. The only 5- and 7-periodic knots with 9 or fewer crossings are the $(5,2)$- and $(7,2)$-torus knots respectively, and no knots with prime periods $q>7$ exist. 
\subsubsection{Knots with 10 crossings} \label{SectionOn10CrossingKnots}
The knot tables in \cite{BZ} also contain information about periodicity of 10 crossing knots, though the knot $10_3$ is listed there incorrectly as having period 3. Odd prime periods $q>3$ have been determined in \cite{N1}.  

Of the 165 knots with 10 crossings, the ones that pass the homological obstructions for 3-periodicity, are 
$$10_4,\, 10_{87},\, 10_{98},\, 10_{99}, \,\boxed{10_{124}},\, 10_{143},  $$
of which only $10_{124}$, being the $(5,3)$-torus knot, is $3$-periodic. Of the alternating knots $10_4, \, 10_{87},\, 10_{98}$ and $10_{99}$, Heegaard Floer correction terms can be used to exclude the knots $10_4$, $10_{87}$ from having period $3$. The knot $10_{98}$ fails to be 3-periodic by the result from \cite{Prz}, and $10_{143}$ can be shown to not be 3-periodic by the methods of \cite{Y1}. The added symmetry of $10_{99}$ coming from its amphicheirality, causes it to pass many 3-periodicity obstructions, including the one stemming from Heegaard Floer correction terms. Indeed the only test that we know that obstructs $10_{99}$ from 3-periodicity , is the one relying on the hyperbolic structure of its complement \cite{AHW}.  

The 10-crossing knots that pass the homological obstructions for 5-periodicity, are 
$$\boxed{10_{123}}, \, \boxed{10_{124}}, \, 10_{132}, \, 10_{137},$$
of which $10_{123}$ and $10_{124}$ are 5-periodic. We already mentioned that $10_{124}$ is the $(5,3)$-torus knot, and as such is 5-periodic, while a 5-periodic diagram of $10_{123}$ is given in Figure \ref{10and11}. The two remaining knots, both non-alternating, can be excluded from 5-periodicity by examining their HOMFLYPT polynomials as in \cite{T3}, or their Jones polynomials as in \cite{M3}. There are no 10-crossing knots with odd prime period $q>5$. 
\begin{figure}[htb!] 
\centering
\includegraphics[width=6cm]{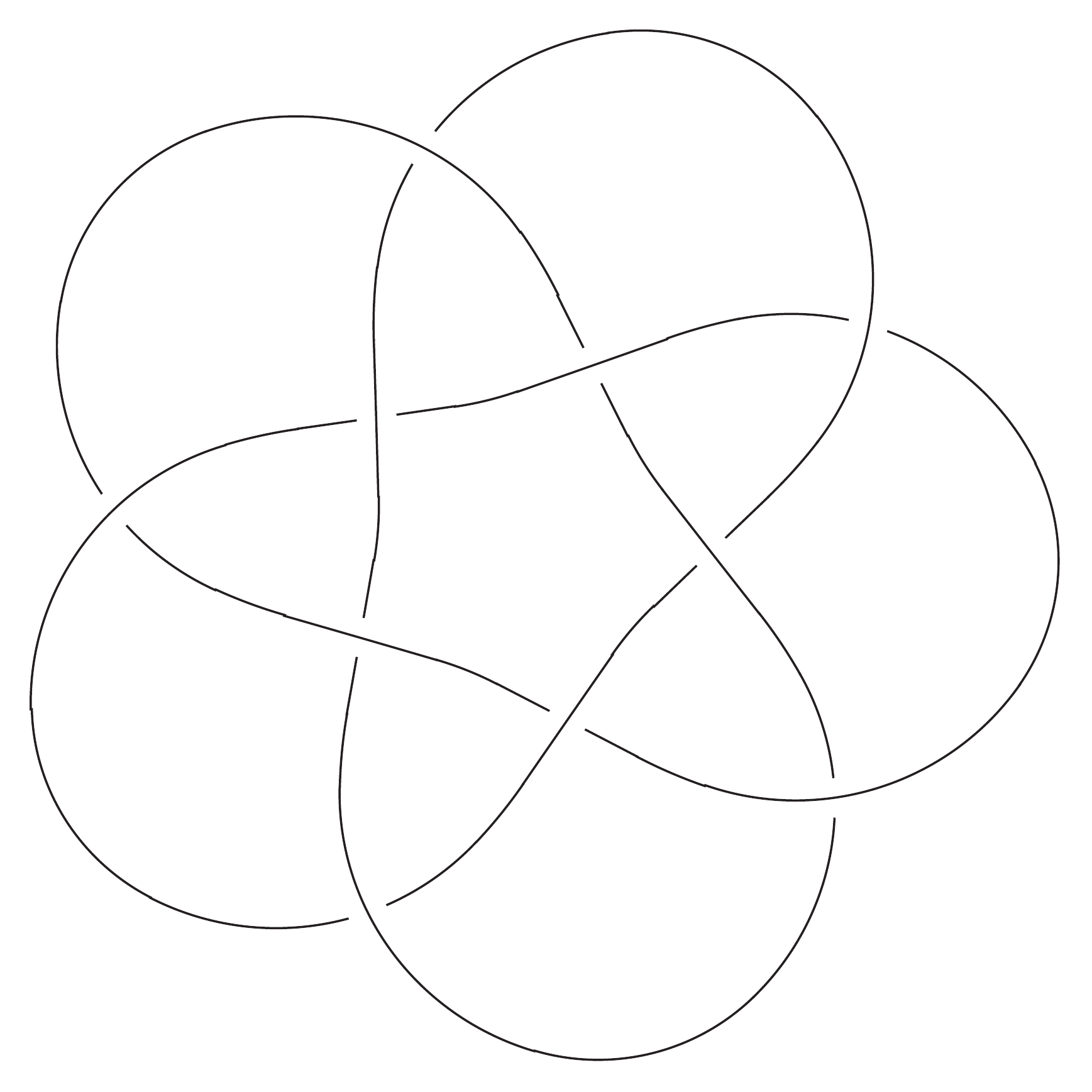}
\caption{Periodic diagram for the knot $10_{123}$.}\label{10and11}
\end{figure} 
\subsubsection{Knots with 11 crossings} \label{SectionOn11CrossingKnots}
Some of the results in this section have previously been obtained for periods $q>3$ in \cite{N1}. Of the 552 knots with 11 crossings, 5 alternating and 9 non-alternating knots pass the homological obstructions for period 3:
$$\begin{array}{l}
11a_{43},\,  11a_{58},\,  11a_{165},\, 11a_{297}, \, 11a_{321},  \cr     
11n_{67},\, 11n_{72}, \, 11n_{77},\,   11n_{97}, 11n_{106}, \, {11n_{126}}, \, {11n_{133}},  \, 11n_{139}, \, 11n_{145}. 
\end{array}
$$ 
Of the alternating knots in the first row, all but $11a_{43}$ are excluded from being 3-periodic by the Heegaard Floer correction terms. $11a_{43}$ is excluded from 3-periodicty by its HOMFLYPT polynomial as in \cite{Prz}. 

Of the non-alternating knots in the second row, all of which are hyperbolic, only the two knots $11n_{126}$ and $11n_{133}$ have full symmetry groups (see Section \ref{SectionOnFullSymmetryGroup} below) with elements of order 3, and are thus the only two knots in this list that may be 3-periodic. They are however excluded from 3-periodicity (and are thus forced to be freely 3-periodic) by their HOMFLYPT polynomials \cite{Prz}. There are no 11-crossing knots with period 3. 

There are no 11-crossing knots with periods $5$ or $7$ either, and there is exactly one 11-periodic 11-crossing knot, namely the $(11,2)$-torus knot $11a_{367}$. There are no 11-crossing knots of prime odd periods $q>11$.  (See~\ref{PeriodCrossings}.)
\subsubsection{Knots with 12 crossings}
We use our techniques to provide a nearly complete classification of 12-crossing alternating knots with odd prime period.  
\begin{thm} \label{PeriodicityForTwelveCrossingAlternatingKnots}
Among the $1288$ alternating 12-crossing knots, there are no examples of knots with odd prime period greater than $3$. There are at most seven $3$-periodic knots, namely 
$$\boxed{12a_{503}}, \,  \boxed{12a_{561}}, \,  \boxed{12a_{615}}, \, 12a_{634}, \,  \boxed{12a_{1019}}, \,  \boxed{12a_{1022}}, \,  \boxed{12a_{1202}}. $$
Of these, the six framed knots are $3$-periodic (Figure \ref{Period3Knots} shows their periodic diagrams) while the remaining knot $12a_{634}$, if $3$-periodic, has a quotient knot with Alexander polynomial $4-7t+4t^{2}$.
\end{thm}
%

\begin{figure}[htb!] 
\centering
\includegraphics[width=14cm]{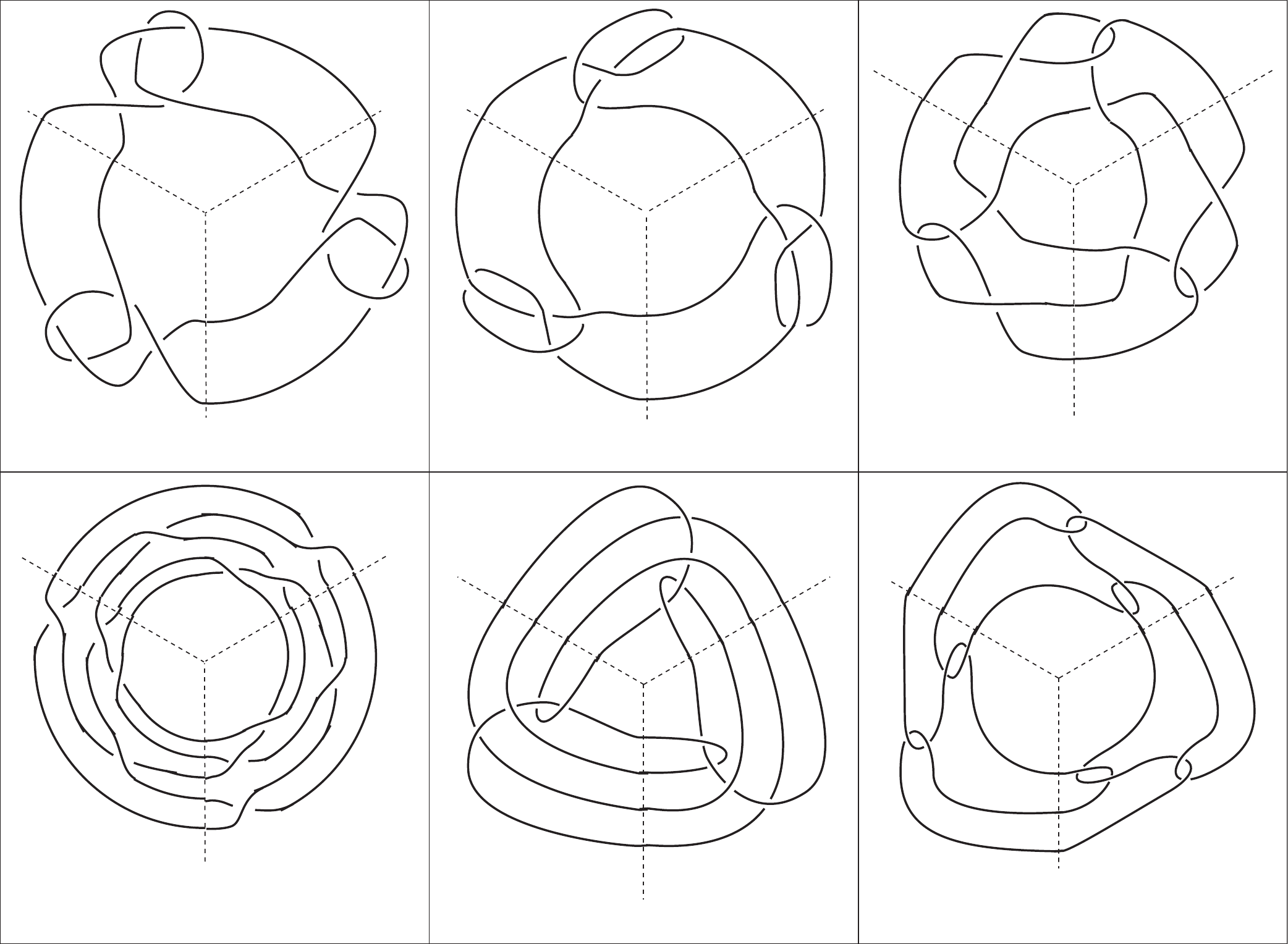}
\put(-350,151){$12a_{503}$}
\put(-220,151){$12a_{561}$}
\put(-85,151){$12a_{615}$}
\put(-350,5){$12a_{1019}$}
\put(-220,5){$12a_{1022}$}
\put(-85,5){$12a_{1202}$}
\caption{Periodic diagrams of certain $3$-periodic 12-crossing alternating knots.}\label{Period3Knots}
\end{figure}
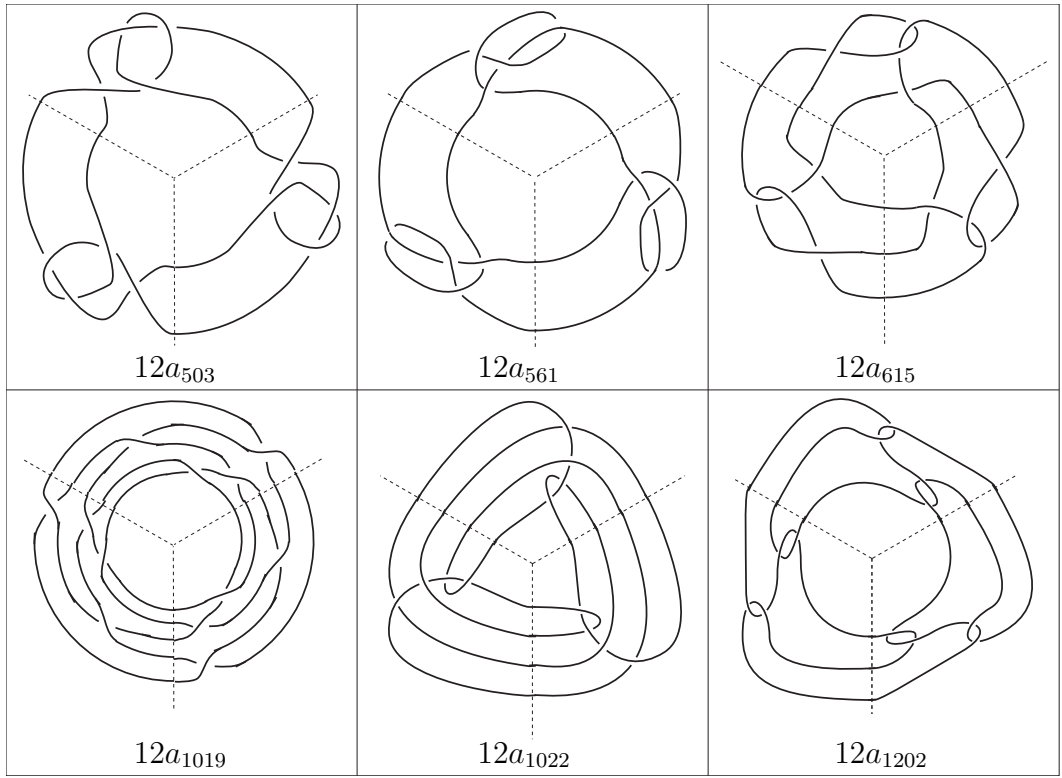 
\subsubsection{Knot with 13, 14 and 15 crossings}
In this category we focus on the alternating knots only, where we are able to give a complete classification of periodic knots of prime periods $q>3$. The exclusion of period $q=3$ from our consideration is due to the large number of  knots that pass the homological obstructions for 3-periodicity.

\begin{thm} \label{131415}
The following provide a complete classification of $q$-periodic alternating knots with 13, 14 or 15 crossings, for $q>3$ a prime. 
\begin{itemize}
\item[(i)] Among the 4,878 alternating knots with 13 crossings, the only knot that possesses a prime period $q>3$ is the $(13,2)$-torus knot $13a_{4878}$. Its only periods are 2 and 13. At most 29 alternating knots with 13 crossings have period 3.
\item[(ii)] Among the 19,536 alternating knots with 14 crossings, the only knot that possesses a prime period $q>3$ is the 7-periodic knot $14a_{19470}$ (see Figure \ref{knot14a19470} for a periodic diagram), and 7 is its only odd prime period.  At most 49 alternating knots with 14 crossings have period 3.
\item[(iii)] Of the 85,263 alternating knots with 15 crossings, the only knots that have a prime period $q>3$ are the knots 
$$  15a_{64035}, \quad  15a_{84903},\quad 15a_{85262},\quad 15a_{85263}.$$ 
Each of these knots has period 5, but no other prime period $q>5$. Periodic diagrams for the first two knots are seen in Figure \ref{15crossingPeriodicKnots}. The knot $15a_{85262}$ is the 5-stranded pretzel knot $P(-3,-3,-3,-3,-3)$, and the knot $15a_{85263}$ is the $(15,2)$-torus knot. At most 133 alternating knots with 15 crossings have period 3.
\end{itemize}
\end{thm}
\begin{figure}[htb!] 
\centering
\includegraphics[width=6cm]{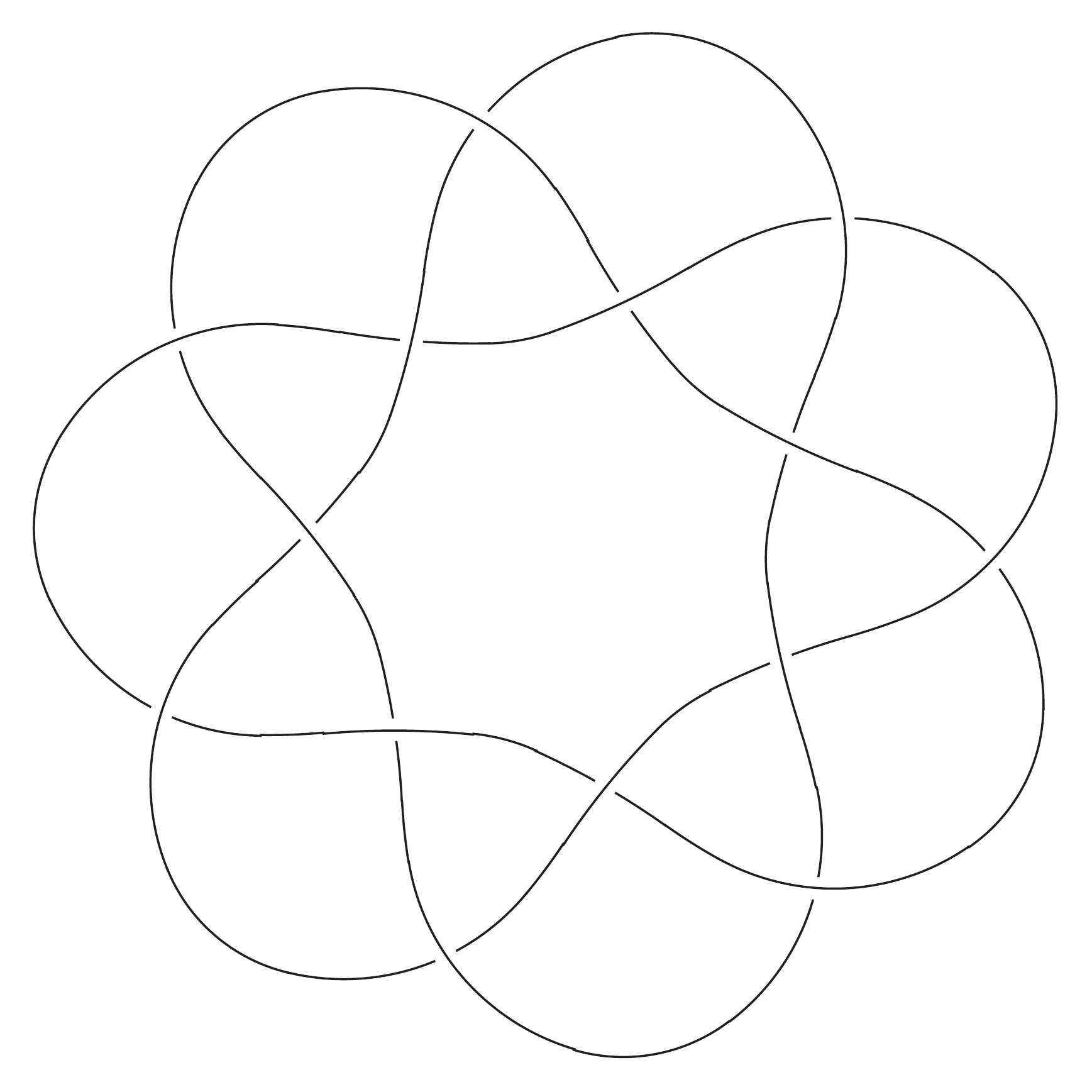}
\caption{A 7-periodic diagram for the knot $14a_{19470}$.}\label{knot14a19470}
\end{figure} 
\begin{figure}[htb!] 
\centering
\includegraphics[width=15cm]{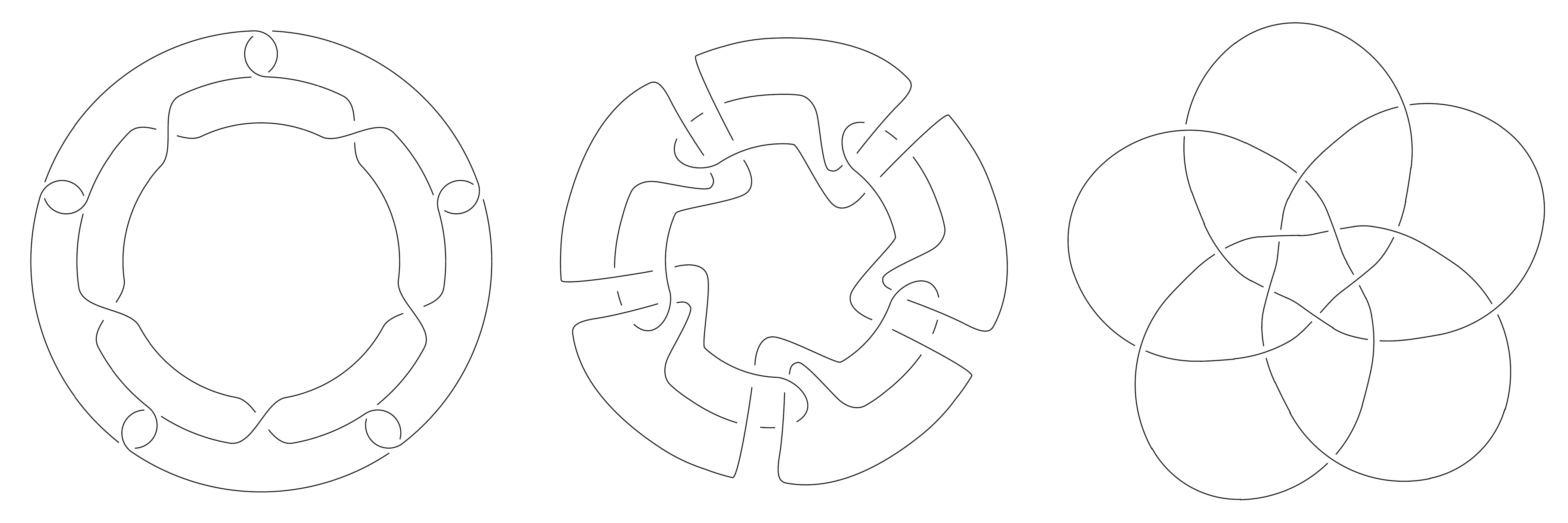}
\put(-370,-20){$15a_{64035}$}
\put(-230,-20){$15a_{64035}$}
\put(-90,-20){$15a_{84903}$}
\caption{Periodic diagrams for the knots $15a_{64035}$ and $15a_{84903}$. The left-most figure gives an alternating 5-periodic diagram for $15a_{64035}$, while the middle figure gives a 5-periodic ribbon diagram for the same knot.}\label{15crossingPeriodicKnots}
\end{figure} 
\subsection{Concluding Remarks and the Full Symmetry Group of a Knot} \label{SectionOnFullSymmetryGroup}
The {\em Full Symmetry Group} $\mathcal S(K)$ of a knot $K$ (see KnotInfo \cite{knotinfo} or Section 10.6 in \cite{kawa}) is the group of diffeomorphisms of the pair $(S^3,K)$, modulo its normal subgroup generated by diffeomorphisms of $(S^3,K)$ isotopic to the identity.  For a hyperbolic knot the symmetry group is always finite, given by the group of isometries of the unique complete hyperbolic structure on the complement of $K$.  Thus, if $K$ is a hyperbolic knot, then $q$-periodicity of $K$ implies the existence of $q$-torsion in $\mathcal S(K)$ (see \cite{AHW} and references therein). 

Conversely however, if $f\in \mathcal S(K)$ is an order $q$ element, there is no guarantee of $q$-periodicity of $K$, even if $K$ is hyperbolic. The reason for this is that $f$ may be fixed point free, rendering $K$ {\em freely $q$-periodic} (a concept not discussed in this work), but not $q$-periodic. For example, $\mathcal S(12n_{276})\cong D_6$ (with $D_6$ being the order $12$ dihedral group) which clearly has elements of order $3$. However, $12n_{276}$ is not $3$-periodic as it does not pass the Murasugi conditions (Section \ref{MurasugiSection}). Other such examples are the hyperbolic knots $9_{48}$, $10_{75}$ and $11a_{366}$, neither of which is $3$-periodic, but each of which has a full symmetry group with elements of order 3. 
\begin{question} \label{QuestionAboutHyperbolicStructure}
Can the hyperbolic structure on the complement of a hyperbolic knot be used to distinguish free periodicity from periodicity of the knot?
\end{question}
In a different direction, a non-hyperbolic knot may be $q$-periodic without $\mathcal S(K)$ having any elements of order $q$. For instance, the torus knots $T_{2,2n+1}$ (Example \ref{PeriodsOfTorusKnots}) have $\mathcal S(T_{2,2n+1})\cong \mathbb Z_2$, but any divisor $q$ of $2n+1$ is a period.  These $q$-periods are thus realized by $f_{q,n}:(S^3,T_{2,2n+1})\to (S^3,T_{2,2n+1}),$ diffeomorphisms that are isotopic to the identity. Recall that torus knots and cable knots are not hyperbolic.

The knot $12a_{634}$ from Theorem \ref{PeriodicityForTwelveCrossingAlternatingKnots} has Full Symmetry Group $\mathcal S(12a_{634}) \cong \mathbb Z_2$ and given that $12a_{634}$ is hyperbolic, it cannot have period $3$. The symmetry groups on KnotInfo were determined by computer calculations as in  \cite{AHW}, and we are not able to eliminate the possibility of period $3$ for $12a_{634}$ without reliance on computers. Besides passing all period $3$ obstructions mentioned in this paper, $12a_{634}$ also passes the obstructions from \cite{D--L, HLN, M1, M3, Prz, T3, Y1, Y2}. 

With regards to the twisted Alexander polynomials obstructions described in \cite{HLN}, the fundamental group of the knot $12a_{634}$ has a unique representation into the dihedral group $D_5$, a representation that is obtained by exploiting the fact that the 2-fold branched cover $Y$ of $12a_{634}$ has $H_1(Y;\mathbb Z)_5 \cong \mathbb Z_5$, and by remembering that $D_5$ is a semidirect product of $\Z _2$ with $\Z _5$. The associated twisted Alexander polynomial of $12a_{634}$ passes the generalized Murasugi condition from \cite{HLN}, and so fails to obstruct 3-periodicity of $12a_{634}$. Despite this, it may well be that twisted polynomials of $K=12a_{634}$ associated to different representations of $\pi_1 (S^3-K)$ could exclude period $3$. 
\begin{conjecture}
If an alternating knot $K$ with crossing number $c$ has odd prime period $q$, then $q$ divides $c$. 
\end{conjecture}
This conjecture holds for $q>3$ for all 111,528 alternating non-trivial knots with crossing number $c \le 15$. The conjecture does not extend to non-alternating knots. For instance, the non-alternating knots $8_{19}$ and $10_{124}$ each have period 3. Also, the $(6,5)$-torus knot, which is non-alternating and has crossing number 24, is 5-periodic.  The conjecture also does not extend to period 2, as any 
torus knot $T_{2,2n+1}$ is alternating with an odd number $(2n+1)$ of crossings and has period 2. 

\vskip3mm
{\bf Organization. } The paper is organized as follows. Section \ref{knotper} discusses knot periodicity. It describes and refines several previously known obstructions to knot periodicity, and revisits Examples \ref{ExampleOf12a100} and \ref{AlgPer} to show that they pass said obstructions. Proofs of Theorem \ref{ObstructionToQPeriodicityWhenFixedPointSetIsTrivial}, \ref{FirstEquivarianceTheorem}, \ref{AbsenceOfL} and \ref{PeriodicityForTwelveCrossingAlternatingKnots} are also provided in this section. Section \ref{SectionWithDetailsOnExamples} provides calculations of the Heegaard Floer correction terms  for the 2-fold brached covers of the knots mentioned in Section \ref{SubSectionOnLowCrossingKnots}. 
\section{Knot periodicity}\label{knotper}
This section begins with background material on knot periodicity and reviews as well as refines some known obstructions to periodicity alluded to in the introduction. The proofs of Theorems \ref{ObstructionToQPeriodicityWhenFixedPointSetIsTrivial},  \ref{FirstEquivarianceTheorem} and \ref{AbsenceOfL} are supplied in Section \ref{ProofsOfTheMainTwoPeriodicityObstructions}. 
\subsection{Background and homological periodicity obstructions} \label{ClassicalPeriodicityObstructions}
Let $K\subset S^3$ be a knot of period $q>1$ and let $f:S^3\to S^3$ be the orientation preserving diffeomorphism realizing $K$'s $q$-periodicity. 

For some $n\in \mathbb N$, let $\wp:Y\to S^3$ be an $n$-fold cyclic covering map of $S^3$ branched along $K$. It is easy to see that $f:S^3\to S^3$ lifts to an order $q$ diffeomorphism $F:Y\to Y$, which in turn induces a $q$-fold cyclic covering map $\Pi:Y\to \overline Y$ (with $\overline Y=Y/\langle F\rangle$) branched over $\bar \wp^{-1}(\overline B)$. The Commutative Diagram \ref{CommutativeDiagramOfCoveringSpaces} captures the descriptions from this paragraph. 
\begin{equation} \label{CommutativeDiagramOfCoveringSpaces}
\xymatrixcolsep{10pc}\xymatrixrowsep{5pc}
\xymatrix{
Y \ar@(u,l)_{F}   \ar[d]_{\tiny \begin{array}{c} \text{$n$-fold cyclic} \cr \text{cover, branched}\cr \text{along } K   \end{array}}^\wp \ar[r]^{{\tiny \begin{array}{c} \text{$q$-fold cyclic cover,} \cr \text{branched along } \bar \wp^{-1}(\overline B)  \end{array}}}_\Pi   & {\overline Y}   \ar[d]^{{\tiny \begin{array}{c} \text{$n$-fold cyclic} \cr \text{cover, branched}\cr \text{along } \overline K   \end{array}}}_{\bar \wp}\\
S^3 \ar@(d,l)^{f} \ar[r]_{{\tiny \begin{array}{c} \text{$q$-fold cyclic cover,} \cr \text{branched along } \overline B   \end{array}}}^\pi & S^3 
}
\end{equation}
With this in place, we turn to several known knot periodicity obstructions. 
\subsubsection{Edmonds' Genus Condition} \label{EdmondsGenusCondition}
In \cite{E} Edmonds showed that if a knot $K$ of genus $g$ has period $q$, then 
\begin{equation} \label{EdmondsCondition}
q\le 2g+1\ {\rm and\ if\ } q>g,\ {\rm then\ } q=g+1\ {\rm or\ } 2g+1. 
\end{equation}  
Furthermore, if the genus of the quotient knot is $\overline g$, then 
\begin{equation}\label{QuotientGenusBound}
g \ge q\overline g.
\end{equation}
Clearly, if $q>g,$ then $\overline g$ is forced to be 0, and therefore $\overline \Delta$ is 1.

If $n$ is an an even number,  the number of disks in a Seifert surface $F$, obtained using Seifert's algorithm on an $n$-crossing knot diagram, is greater than 2, as only the trivial knot has exactly one Seifert circuit and only the torus knots of type $(n,2)$ have exactly two Seifert circuits and $n$ in that case has to be odd.  The Euler characteristic of such an $F$ is $1-2g(F) \ge -n+3$ 
and the genus of the knot $g \le g(F)$.  It follows that (see~\cite[Lemma 4.5]{N1}) if $K$ is an $n$-crossing knot with $n$ even, then 
\begin{equation} \label{GenusCrossings}
g \le n/2 - 1. 
\end{equation} 
Combining,~\ref{EdmondsCondition} and~\ref{GenusCrossings}, we have $q < 2g+1 \le n-1.$
On the other hand, if $n$ is odd, the $T(2,n)$ torus knot has period n.  So we have, for an $n$-crossing knot 
\begin{equation} \label{PeriodCrossings}
q \le n,\ {\rm if\ }n\ {\rm is\ odd,\ and\ } q \le n-1,\ {\rm if\ }n\ {\rm is\ even.\ }
\end{equation}

These conditions substantially limit the periods a given knot may possess. 
\subsubsection{Murasugi's Alexander Polynomial Conditions} \label{MurasugiSection}
Note that the Alexander polynomial of a knot is a polynomial $\Delta$ with integer coefficients, such that 
\begin{equation} \label{AlexSymmetry}
\Delta(t) \, =\, t^{{\rm deg}(\Delta)}\Delta(t^{-1}).
\end{equation}

Let $K$ be a $q$-periodic knot with quotient knot $\overline K$ and let $\lambda = |\ell k (K,B)|$ be the absolute value of the linking number of $K$ with its axis $B$. Let $\Delta _K(t)$ and $\Delta _{\overline K}(t)$ be the Alexander polynomials of $K$ and $\overline K$ respectively.  The next two conditions constrain periodicity of $K$ with $\overline K$ as a quotient knot. 
\begin{equation}\label{Murasugi1}
\Delta_{\overline K}\, \vert\, \Delta_K \quad \text{ in }  \mathbb Z[t,t^{-1}].
\end{equation}
\begin{equation} \label{Murasugi2}
\Delta_K (t) \stackrel{\cdot}{\equiv} \left( \Delta_{\overline K}(t)\right)^q \left( 1+t+\dots +t^{\lambda -1}\right)^{q-1} \, (\text{mod } q).
\end{equation}
The symbol \lq\lq $\stackrel{\cdot}{\equiv}$\rq\rq \, stands for congruence modulo $q$ up to multiplication by units in $\mathbb Z[t,t^{-1}]$. Additionally, $\gcd(\lambda , q ) = 1$. The proofs of these can be found in \cite{M1,M2}. 
Conditions \eqref{Murasugi1} and \eqref{Murasugi2} are  called the {\em Murasugi Conditions}.  

The genus of any knot is necessarily greater than or equal to half of the degree of its Alexander polynomial.   Therefore, using  \eqref{QuotientGenusBound} we have
\begin{equation} \label{GenusAlexander}
g(K) \ge \frac{q}{2}{\rm deg}({\Delta}_{\overline K}).
\end{equation}
\subsubsection{The Homology Condition}  \label{SectionHomologyCondition}
Recall from Remark \ref{qln} that $q$ and $\ell$ denote distinct primes and $n$ is a prime-power. 
Let  
$g_q(\ell)$ be the smallest natural number such that  
\begin{equation}\label{DefinitionOfGql}
 \ell ^{g_q(\ell)} \equiv \pm 1 \,  (\text{mod } q).
\end{equation}
Let $Y$ and $\overline Y$ be $n$-fold cyclic covers of $S^3$ branched over a  $q$-periodic knot $K$ and its quotient knot $\overline K$ respectively. Then there exist non-negative integers  $s,b_1,\dots b_s$ such that, after identifying $H_1(\overline Y;\mathbb Z)_\ell $ with a subgroup of $H_1(Y;\mathbb Z)_\ell$ as allowed by Equation \eqref{FixedPointSetOfFStar}, we have the following isomorphism, proved in \cite{N3}, and henceforth referred to as the {\em Homology Condition}.  
\begin{equation} \label{StructureOfQuotientOfFirstHomologies2}
H_1(Y;\mathbb Z)_\ell / H_1(\overline Y;\mathbb Z)_\ell  \cong \mathbb Z_\ell ^{2b_1 g_q(\ell)} \oplus \mathbb Z_{\ell^2} ^{2b_2 g_q(\ell)} \oplus \dots \oplus \mathbb Z_{\ell^s} ^{2b_s g_q(\ell)}. \end{equation}

Prime-power fold cyclic covers $Y$ of $S^3$ branched over a knot are rational homology spheres, and the degree of the cover is relatively prime to the order of the first homology of $Y$. If the degree of the cover is odd, then $H_1(Y;\mathbb Z)$ is always a double (that is $H_1(Y;\mathbb Z)\cong G \oplus G$ for some $G$),  but in general this is not the case for even-fold covers. In Corollary \ref{2^n} below we observe that for a periodic knot we have a ``double'' in homology irrespective of the parity of the prime-power fold cover.

\begin{cor}\label{2^n}
Let $q$ and $\ell$ be two distinct primes and $K$ a $q$-periodic knot with $Y$ its $n$-fold cyclic branched cover (with $n$ a prime-power of arbitrary parity). Then 
\begin{equation}  \label{Known1}
H_1(Y;\mathbb Z)_{\ell}/H_1(\overline{Y};\mathbb Z)_{\ell} \mbox{ is a double.} 
\end{equation}

In particular if $Y$ is the double branched cover of $K$, $\Delta_{\overline K}(t)$ the Alexander polynomial of the  quotient knot, and $\ell \not\vert{\Delta_{\overline K}}(-1),$ then 
\begin{equation}\label{double}
H_1(Y;\mathbb Z)_{\ell} \mbox{ is a double.} 
\end{equation}
\end{cor}
\begin{proof}
The claim \eqref{Known1} is a direct consequence of 
\eqref{StructureOfQuotientOfFirstHomologies2} as 
$\mathbb Z_{\ell^m} ^{2b g_q(\ell)} \cong \mathbb Z_{\ell^m} ^{b g_q(\ell)}\oplus \mathbb Z_{\ell^m} ^{b g_q(\ell)}$, while \eqref{double}  is implied by \eqref{Known1} along with the observation that if $\ell$ does not divide $\Delta_{\overline K}(-1)$, then $H_1(\overline Y;\mathbb Z)_\ell = 0$. With regards to \eqref{Known1} and \eqref{double}, we note that the trivial group is a double. 
\end{proof}

As mentioned in Section \ref{SubSectionOnLowCrossingKnots} we shall refer to the periodicity obstructions from Section \ref{EdmondsGenusCondition} (Edmonds' Genus Conditions), Section \ref{MurasugiSection} (Murasugi's Alexander Polynomail Conditions) and from the present section (the Homology Condition) collectively as the {\em homology periodicity obstructions}.  When in Subsection \ref{ExamplesOfPeriodicAndNonPeriodicKnots} we revisit Examples \ref{ExampleOf12a100} and \ref{AlgPer} from the introduction, we shall see that that the knots considered therein pass each of the obstructions \, \eqref{EdmondsCondition}, \eqref{Murasugi1}, \eqref{Murasugi2}, \eqref{GenusAlexander} and \eqref{StructureOfQuotientOfFirstHomologies2}, underscoring the strength of the Heegaard Floer obstruction. 
\subsection{Proofs of Theorems \ref{ObstructionToQPeriodicityWhenFixedPointSetIsTrivial}, \ref{FirstEquivarianceTheorem} and \ref{AbsenceOfL}} \label{ProofsOfTheMainTwoPeriodicityObstructions}
For the two proofs in this section, we rely on the following notation and assumptions: 
\vskip1mm
Let $K$ be a $q$-periodic knot whose periodicity is realized by an order $q$, orientation preserving diffeomorphism $f:S^3\to S^3$. For some prime-power $n$, let $Y$ be the $n$-fold cyclic cover of $S^3$ branched over $K$ and let $F:Y\to Y$ be the lift of $f$. Note that $Y$ is a rational homology $3$-sphere. Let $\overline Y$ be the $n$-fold cyclic cover of $S^3$ branched along the quotient knot $\overline K$ and let $\ell$ be a prime distinct from $q$. Recall also that we tacitly identity $Spin^c(Y)$ with $H_1(Y;\mathbb Z)$ through an $F$-compatible affine identification (see Remark \ref{FCompatibility}), and write $s$ or $F_*(s)$ with $s\in H_1(Y;\mathbb Z)$ instead of $\s$ or $F^*(\s)$ with $\s \in Spin^c(Y)$.
\subsubsection{Proof of Theorem \ref{ObstructionToQPeriodicityWhenFixedPointSetIsTrivial}} 
The hypothesis of Theorem \ref{ObstructionToQPeriodicityWhenFixedPointSetIsTrivial} states that $H_1(\overline Y;\mathbb Z)_\ell = 0$, showing that the fixed point set of $F_*$ restricted to $H_1(Y;\mathbb Z)_\ell$ is $\{0\}$ (as follows from \eqref{FixedPointSetOfFStar}). Thus for every non-zero element $s\in H_1(Y;\mathbb Z)_\ell$, the set  $\{F_*(s), F_*^2(s), \dots , F_*^q(s)\}$ contains $q$ distinct spin$^c$-structures. 
As $d(Y,F^{i_1}_*(s)) = d(Y,F^{i_2}_*(s))$ for any pair $i_1,i_2\in \{1,\dots,q\}$ (Theorem \ref{HeegaardFloerTheoremOnPullbackOfCorrectionTerms}), the claim of Theorem \ref{ObstructionToQPeriodicityWhenFixedPointSetIsTrivial}  follows. \hfill $\Box$ 
\subsubsection{Proof of Theorem \ref{FirstEquivarianceTheorem}} 
Define $H$ to be the subgroup of $H_1(Y;\mathbb Z)_\ell$ given by $H=\text{Fix}(F_*|_{H_1(Y;\mathbb Z)_\ell})$, and note that $H\cong H_1(\overline Y;\mathbb Z)_\ell$ according to \eqref{FixedPointSetOfFStar}. Consider the cosets $s+H$ of $H$ in $H_1(Y;\mathbb Z)_\ell$ and for any $s+h\in s+H\ne 0+H$, consider the associated orbit $\{s+h,F_*(s+h), F^2_*(s+h),\dots,F^{q-1}_*(s+h)\}$ of $s+h$ under the action of $\mathbb Z_q=\{\text{Id},F_*,\dots,F_*^{q-1}\}$. Each such orbit has $q$ elements since $q$ is prime (restricting the cardinality of said orbit to be either $1$ or $q$) and because an orbit of cardinality $1$ would lead to $s+h = F_*(s+h)$ and thus to $s+h\in H$, contrary to assumption. Accordingly, there are 
$$\frac{([H_1(Y;\mathbb Z):H]-1)\cdot |H|}{q}$$
values of correction terms (not necessarily all distinct) $d(Y,s)$ with $s\in H_1(Y;\mathbb Z)_\ell$, each of which occurs with multiplicity $q$. The remaining $|H|$ correction terms $d(Y,s)$ with $s\in H$, may have arbitrary multiplicity. Theorem \ref{FirstEquivarianceTheorem} now follows.  \hfill $\Box$ 
\subsubsection{Proof of Theorem \ref{AbsenceOfL}} \label{ProofOfProposition19}
Assume $K$ is a $q$-periodic knot with $q$ a prime, let $\overline K$ be its quotient knot and let $\Delta _K(t)$ and $\Delta _{\overline K}(t)$ be their respective Alexander polynomials. Let $Y$ and $\overline Y$ be the $n$-fold cyclic branched covers of $S^3$ with branching sets $K$ and $\overline K$ respectively, with $n$ a power of a prime. Let $F_{*,q} = F_*|_{H_1(Y;\mathbb Z)_q}$.  The proof of the theorem rests on the existence of a \lq\lq {\em transfer map}\rq\rq \, $\mu_*:H_1(\overline Y;\mathbb Z)_q \to H_1(Y;\mathbb Z)_q$ with the property that if $F_{*,q}(\alpha ) = \alpha$ for some $\alpha \in H_1(Y;\mathbb Z)_q$, then (see Diagram \ref{CommutativeDiagramOfCoveringSpaces} for a definition of $\Pi$) 
\begin{equation} \label{RelationForTransferMap}
\mu_* \circ \Pi_* (\alpha )  = q \cdot \alpha .
\end{equation}
The existence of $\mu_*$ and the validity of relation \eqref{RelationForTransferMap} follow from the results in Section III.2 of \cite{Bredon}, see specifically relation (2.2) in said section. 

Write $Fix(F_{*,q})$ to mean the fixed point set of $F_{*,q}:H_1(Y;\mathbb Z)_q \to H_1(Y;\mathbb Z)_q$, and consider the the commutative diagram
\begin{equation} \label{CommutativeDiagramWithTransferMap}
\xymatrixcolsep{10pc}\xymatrixrowsep{5pc}
\xymatrix{
 Fix(F_{*,q}) \ar[dr]_{\Pi_*}  \ar[rr]^{\cdot q}  &  & H_1(Y;\mathbb Z)_q   \\
&   H_1(\overline Y;\mathbb Z)_q  \ar[ur]_{\mu_*}&  
}
\end{equation}
If $Fix(F_{*,q})\cong \mathbb Z_q^{m_1}\oplus \dots \oplus \mathbb Z_{q^k}^{m_k}$, then image of the horizontal map in \eqref{CommutativeDiagramWithTransferMap} is isomorphic to $\mathbb Z_q^{m_2}\oplus \dots \oplus \mathbb Z_{q^{k-1}}^{m_k}$ and is contained in the image of $\mu_*$. Thus, 
$$
|Fix(F_{*,q})|  = q^{m_1+\dots +m_k}\cdot |Im(\cdot q) | \le q^{m_1+\dots +m_k}  \cdot |Im(\mu_*)| \le q^{m_1+\dots +m_k}\cdot |H_1(\overline Y;\mathbb Z)_q|,
$$
as claimed in Theorem  \ref{AbsenceOfL}. Corollary \ref{CorollaryForTheCaseOfQPrimaryGroups} is an easy consequence of Theorem \ref{AbsenceOfL}. 
\hfill $\Box$ 
\subsection{Examples \ref{ExampleOf12a100} and \ref{AlgPer} revisited} \label{ExamplesOfPeriodicAndNonPeriodicKnots}
This section shows that the knots in Examples \ref{ExampleOf12a100} and \ref{AlgPer} pass the homological  periodicity obstructions from Sections \ref{EdmondsGenusCondition} -- \ref{SectionHomologyCondition}, as claimed in the introduction. This underscores the usefulness of the periodicity obstruction coming from the Heegaard Floer correction terms.
\subsection{Calculations for Example \ref{ExampleOf12a100}}\label{Pf12a100}
Let $K$ be the knot $K=12a_{100}$. Its Alexander polynomial $\Delta_K(t)$ and the factorization of $\Delta_K(t)$ over $\mathbb Z$ into irreducible factors, is given by 
\begin{align} \label{12a100FirstMurasugiCondition}
\Delta_K(t) & = 3t^{6}-21t^{5}+ 53t^{4}-71t^3+ 53t^2-21t+ 3,\cr
                   & = (t^3-5t^2+6t-3)(3t^3-6t^{2}+5t-1).
\end{align} 
The Murasugi Condition \eqref{Murasugi2} for $\Delta_K(t)$ and with $q=3$, reads
\begin{align} \label{12a100SecondMurasugiCondition}
(1+t)^2 &  \stackrel{\cdot}{\equiv} (1+t+\dots+t^{\lambda -1})^2 \cdot (\Delta_{\overline K}(t))^3 \, \,(\text{mod } 3).
\end{align}
As by \eqref{AlexSymmetry} and \eqref{Murasugi1} $\Delta_{\overline K}$ is a symmetric factor of $\Delta_K$, 
the only possibilities for $\Delta_{\overline K}(t)$ are $1$ or $\Delta _K (t).$ Of these, the only one that fits condition \eqref{12a100SecondMurasugiCondition} is $\Delta_{\overline K}(t) = 1$. Thus $K=12a_{100}$ passes the Murasugi condition with $q=3$, $\Delta_{\overline{K}}(t) =1$ and $\lambda  = 2$ (note that, as required, $\gcd (\lambda , q)=1$).   

The first homology of the $2$-fold cyclic cover $Y$ of $S^3$ branched over $K$ is 
$$H_1(Y;\mathbb Z) \cong \mathbb Z_{5}\oplus \mathbb Z_{5}\oplus \mathbb Z_{9}.$$
Given this and given $q=3$, the only meaningful choice of $\ell$ is $5$. Since $g_3(5) = 1$ (see \eqref{DefinitionOfGql}), condition \eqref{StructureOfQuotientOfFirstHomologies2} gives us
 (keeping in mind that $\Delta_{\overline K}(t)=1$ implies $H_1(\overline Y;\mathbb Z) = 0$) 
\begin{align*}
H_1(Y;\mathbb Z)_5 & \cong \mathbb Z_5^{2b_1}\oplus \mathbb Z_{25}^{2b_2}\oplus \dots \oplus \mathbb Z_{5^t}^{2b_t},
\end{align*}
which is clearly satisfied with $t=1$, $a_1=1$ and $a_i=0$ for $i\ge 2$. Accordingly, $K=12a_{100}$ passes the Homology Condition.  
\subsection{Calculations for Example \ref{AlgPer}} \label{PfAlgPer} 
Let $K=7_4\#7_4\#9_2$ be as in Example \ref{AlgPer}. As already noted, the knots $K_1=7_4$ and $K_2=9_2$ have the same Alexander polynomial $\Delta_{K_i}(t) = 4t^2 -7t + 4$, so that the Alexander polynomial of $K$ is $\Delta _K(t) = (4t^2 -7t + 4)^3$.  As $4t^2 -7t + 4$ is irreducible over $\mathbb Z$, the only possibilities for the Alexander polynomial of a quotient knot $\overline K$ are $1$ or powers of $4t^2 -7t + 4.$ 
The Murasugi condition \eqref{Murasugi2} with $q=3$ becomes 
$$ (1+t)^6 \stackrel{\cdot}{\equiv} (1+t+\dots + t^{\lambda -1})^2\cdot (\Delta_{\overline K}(t))^3 \, (\mbox{mod }3),$$
which forces $\Delta_{\overline K}(t) = 4t^2 -7t + 4$ and $\lambda =1$. With theses choices, $K$ satisfies the Murasugi Conditions \eqref{Murasugi1} and \eqref{Murasugi2}. 

It is easily seen that $K$ also  meets the Homology Condition 
\eqref{StructureOfQuotientOfFirstHomologies2} since 
$$H_1(Y_i;\mathbb Z) \cong \mathbb Z_3\oplus \mathbb Z_5,$$
(where $Y_i$ is the $2$-fold cyclic cover of $S^3$ branched along $K_i$, $i=1,2$), rendering $H_1(Y;\mathbb Z)$ (with $Y$ the $2$-fold cyclic cover of $S^3$ branched along $K$) isomorphic to $H_1$ of the $2$-fold cyclic branched cover of $S^3$ branched along the $3$-periodic knot $K'=7_4\#7_4\#7_4$.  

We invite the reader to check that the knot $7_4\# 9_2\#9_2$ is also excluded from being 3-periodic by the Heegaard Floer correction terms (though it too passes the homological 3-periodicity obstructions). 
\section{Calculations for low-crossing knots} \label{SectionWithDetailsOnExamples}

This section supplies background on the massive computational effort on which the results from Section \ref{SubSectionOnLowCrossingKnots} are based. All our correction term calculations for alternating knots are based on an implementation of the Ozsv\'ath-Szab\'o algorithm from \cite{peter14} into Mathematica. Likewise, the three homological knot periodicity obstructions from Sections  \ref{EdmondsGenusCondition} -- \ref{SectionHomologyCondition}, have also been conducted with computations in Mathematica. We do not address non-alternating knots as at present there is no efficient algorithm for the computation of correction terms of their 2-fold cyclic branched covers. As in the introduction, by homological $q$-periodicity obstructions we mean the $q$-periodicity obstructions of Edmonds, Murasugi and the Homology obstruction, discussed in Sections \ref{EdmondsGenusCondition}, \ref{MurasugiSection} and \ref{SectionHomologyCondition} respectively.  

{\em Notational conventions. } For the remainder of this section, when analyzing potentially periodic knots, we denote them by $K, K_i, L,L_i$, we label their hypothetical quotient knots by $\overline K, \overline K_i, \overline L, \overline L_i$,  and their respective 2-fold cyclic branched covers by $Y, Y_i, X, X_i$ and $\overline Y, \overline Y_i, \overline X, \overline X_i$. For a knot $M$, $\Delta _M(t)$ denotes its Alexander polynomial, and the symbol $\stackrel{.}{\equiv}$ means congruence up to multiplication by $\pm t^n$, $n\in \mathbb Z$. 
\subsection{Alternating knots with up to eleven crossings}\label{11CrPeriodicEx} 
With the exception of the period 3 results for 11-crossing knots, all results from Sections \ref{SectionOn9CrossingKnots} -- \ref{SectionOn11CrossingKnots} have appeared elsewhere in the literature. We thus focus on our claimed obstruction for 3-periodicity for 11-crossing alternating knots.  Recall that the only 5 knots which pass the homological 3-periodicity obstruction are $11a_{43},\,  11a_{58},\,  11a_{165},\, 11a_{297}, \, 11a_{321}$. All but the first one is excluded from actually being 3-periodic by the Heegaard Floer correction terms. 

If $K_1$ is either $11a_{58}$ of $11a_{165}$, then $\Delta _{K_1}(t) = (t-2)(2t-1)(1-t+t^2)^2\stackrel{.}{\equiv} (1+t)^6 \,(\text{mod } 3)$. Accordingly, $K_1$ passes the Murasugi condition with hypothetical quotient knot $\overline K_1$ with either $\Delta_{\overline K_1}(t) = (t-2)(2t-1)$ or $\Delta_{\overline K_1}(t) = 1-t+t^2$. We find that $|H_1(\overline Y_1;\mathbb Z)|$ equals either 3 or 9. As $H_1(Y_1;\mathbb Z)\cong \mathbb Z_{81}$, Corollary \ref{CorollaryForTheCaseOfQPrimaryGroups} implies that there are at least 81-27 = 54 correction terms of $Y_1$ which occur with multiplicities divisible by 3. This is not the case for either choice of $K_1$, as evidenced by the table (while $11a_{58}$ and $11a_{165}$ share this table of multiplicities, they do not have the same correction terms overall).

\vskip2mm
\begin{center}
\begin{tabular}{|c|c||c|c|c|c|c|} \hline 
\multirow{2}{*}{$11a_{58}$ or $11a_{165}$}&$d(Y,s)$  & $-\frac{8}{9}$  & -$\frac{2}{9}$  &  $0$ & $\frac{4}{9}$ & All others  \cr \cline{2-7} 
&Multiplicity of $d(Y,s)$ & $6$ &  $6$ & $9$ & $6$ & $2$\cr \hline 
\end{tabular}
\end{center}
\vskip2mm

For $K_2=11a_{297}$ we obtain $\Delta_{K_2}(t) = (1-3t+t^2)^2(2-3t+2t^2) \stackrel{.}{\equiv} (1+t^2)^3 \,(\text{mod } 3)$, which passes the Murasugi polynomial condition with quotient knot $\overline K_2$ with Alexander polynomial $\Delta_{\overline K_2}(t) = 1-3t+t^2$ or $\Delta_{\overline K_2}(t) = 2-3t+2t^2$. As the first homology of $Y_2$ is given by $H_1(Y_2;\mathbb Z) \cong \mathbb Z_5\oplus \mathbb Z_5\oplus \mathbb Z_7$, the possibility  $\Delta_{\overline K_2}(t)=1-3t+t^2$ is excluded by the homology condition, leaving $\Delta_{\overline K_2}(t) = 2-3t+2t^2$. This leads to $H_1(Y_2;\mathbb Z)_5/H_1(\overline Y_2;\mathbb Z)_5 \cong \mathbb Z_5\oplus \mathbb Z_5$, and thus Theorem \ref{ObstructionToQPeriodicityWhenFixedPointSetIsTrivial} with $\ell = 5$ implies the existence of 24 correction terms $d(Y_2,s)$ with multiplicities divisible by 3, a condition that is violated. 
\vskip2mm
\begin{center}
\begin{tabular}{|c|c||c|c|c|c|} \hline 
\multirow{2}{*}{$11a_{297}$}&$d(Y,s)$  & $-\frac{13}{10}$, \, $\frac{11}{10}$  & -$\frac{9}{10}$, \, $\frac{7}{10}$  &  $-\frac{1}{2}$ & $-\frac{1}{10}$,\, $\frac{3}{10}$   \cr \cline{2-6} 
&Multiplicity of $d(Y,s)$ & $2$ &  $4$ & $1$ & $6$ \cr \hline 
\end{tabular}
\end{center}
\vskip2mm

For $K_3=11a_{321}$ we obtain $\Delta_{K_3}(t) = 3-15t+27t^2-31t^3+27t^4-15t^5+3t^6\stackrel{.}{\equiv} 1 \, (\text{mod } 3)$, passing Murasugi's condition with quotient knot $\overline K_3$ with $\Delta _{\overline K_3}(t) = \Delta_{K_3}(t)$ or $\Delta _{\overline K_3}(t) =1$. The former possibility is excluded by Edmonds' condition, forcing $\Delta _{\overline K_3}(t) =1$. $Y_3$ has first homology $H_1(Y_3;\mathbb Z) \cong \mathbb Z_{11}\oplus \mathbb Z_{11}$, so that Theorem \ref{ObstructionToQPeriodicityWhenFixedPointSetIsTrivial} with $\ell = 11$, implies the existence of 120 correction terms $d(Y_3,s)$ with multiplicities divisible by 3. This is not the case: Each of the values $\frac{\lambda }{11}$ with $\lambda \in \{-7, \pm 5, \pm 3, \pm 1\}$ equals a correction term  of $Y_3$ with multiplicity 12, each of $\frac{-9}{11}$ and $\frac{7}{11}$ is the value of a correction term with  multiplicity 10, $\frac{9}{11}$ comes with multiplicity 8, and the remaining correction terms have multiplicities 1, 2 and 4. Thus, there are at most 102 correction terms with multiplicities divisible by 3, preventing 3-periodicity of $11a_{321}$. 

There are no 11-crossing knots that pass the homological $q$-periodicity obstructions for a prime $q>3$, save $11a_{367}$ which, being the $(11,2)$-torus knot, passes the 11-periodicity obstructions. 
\subsection{Twelve crossing alternating knots}\label{12CrPeriodicEx} 
It follows from~\ref{GenusCrossings} that the largest possible genus of any knot with 12 crossings is $g=5$ and from~\ref{PeriodCrossings}  that the largest possible period is 
$q=11$.  Thus, we proceed by checking which 12-crossing alternating knots pass the homological $q$-periodicity obstructions for $q=3,5,7,11$. 

There are 17 knots that pass the homological 3-periodicity obstructions, namely the six 3-periodic knots $12a_{503},\, 12a_{561},\, 12a_{615},\, 12a_{1019},\, 12a_{1022},\, 12a_{1202}$ from Figure \ref{Period3Knots}, and the eleven knots listed below:
\begin{equation} \label{SpecialKnotsNeedingExtraTreatment}
\begin{array}{l}
 12a_{100},\, 12a_{348},\, 12a_{376},\, 12a_{1206}, \cr
 12a_{390},\, 12a_{425},\, 12a_{459},\, 12a_{596}\, 12a_{672}, \cr
 12a_{634}, \, 12a_{780}.
\end{array}
\end{equation}
Of the knots in the first row of \eqref{SpecialKnotsNeedingExtraTreatment}, $K_1=12a_{100}$ was already shown to not have period $3$ in Example \ref{ExampleOf12a100} using correction terms. Using the same method, we shall demonstrate next that the other three knots in this row also cannot have period $3$. 

The Alexander polynomial of knot $K_2=12a_{348}$ is 
\begin{align}  \label{FactorsOf12a348}
\Delta_{K_2}(t)  & =  2-17t+ 54t^{2}-79t^3+ 54t^4-17t^5+ 2t^6,  \cr
& = (t-2) (2t-1) (t^2 - 3t + 1)^2,
\end{align}
and its mod 3 reduction is 
\begin{align}
\Delta_{K_2}(t) \stackrel{\cdot}{\equiv}  (1+t+\dots+t^{3})^2\, \, (\mbox{mod } 3). 
\end{align}
This with \eqref{Murasugi2} forces  $\lambda = 4$ and $\Delta_{\overline K_2}(t) \stackrel{\cdot}{\equiv} 1$. Examining the factors of $\Delta_{K_2}(t)$ in \eqref{FactorsOf12a348}, we find that $\Delta_{\overline K_2}(t) = 1$. Accordingly, the $H_1(\overline Y_2;\mathbb Z)=0$. Choosing $\ell = 5$, Theorem \ref{ObstructionToQPeriodicityWhenFixedPointSetIsTrivial} shows that all correction terms $d(Y_2,s)$ with $s\in H_1(Y_2;\mathbb Z)_5-\{ 0\}$ must come with multiplicities divisible by $3$. However, the correction terms and their multiplicities in the next table show that this is not the case.  
\vskip2mm
\begin{center}
\begin{tabular}{|c|c||c|c|c|c|c|c|c|c|c|} \hline 
\multirow{2}{*}{$12a_{348}$}&$d(Y_2,s)$  & $-\frac{94}{45}$  & $-\frac{58}{45}$  &  $-\frac{8}{9}$ & $-\frac{28}{45}$ & $-\frac{16}{45}$ & $-\frac{2}{9}$  &$\frac{2}{45}$  & $\frac{14}{45}$ & $\frac{4}{9}$ \cr \cline{2-11} 
&Multiplicity of $d(Y_2,s)$ & $1$ &  $2$ & $4$ & $2$ & $2$ & $2$ & $2$ & $6$ & $4$ \cr \hline 
\end{tabular}
\end{center}
\vskip2mm

Turning to $K_3=12a_{376}$, we note that  $H_1(Y_3;\mathbb Z) \cong \mathbb Z_{27}\oplus \mathbb Z_5$ while its quotient knot $\overline K_3$, if it exists, has $H_1(\overline Y_3;\mathbb Z)\cong \mathbb Z_{15} \cong \mathbb Z_{3}\oplus \mathbb Z_5$. Envoking Corollary \ref{CorollaryForTheCaseOfQPrimaryGroups} with $q=3$, we see that $|Fix(F_{*,3})|\le 9$, showing that there are at least 18 correction terms $d(Y,s)$, $s\in H_1(Y_3;\mathbb Z)_3\cong \mathbb Z_{27}$ that come with a multiplicity divisible by 3. An explicit computation shows that this is not the case: 
\vskip2mm
\begin{center}
\begin{tabular}{|c|c||c|c|c|} \hline 
\multirow{2}{*}{$12a_{376}$} & 
$d(Y_3,s)$  & 
$\frac{1}{2}$  & 
$-\frac{1}{6}$  & 
All others  \cr \cline{2-5} 
 & 
 Multiplicity of $d(Y_3,s)$ & 
 $3$ &  
 $6$ &  
 $2$ \cr \hline 
\end{tabular}
\end{center}
\vskip2mm

The knot $K_4=12a_{1206}$ from \eqref{SpecialKnotsNeedingExtraTreatment} has  $H_1(Y_4;\mathbb Z) \cong \mathbb Z_7\oplus \mathbb Z_{35}$ and if there is a quotient knot $\overline K_4$, it has $H_1(\overline Y_4;\mathbb Z) \cong \mathbb Z_5$. This allows for an application of Theorem \ref{FirstEquivarianceTheorem} with the choice of $\ell = 7$, forcing the $48$ correction terms $d(Y_4,s)$ of $Y_4$ corresponding to $s\in H_1(Y_4;\mathbb Z)_7-\{0\}$ to come with values with multiplicities divisible by $3$. An explicit computation shows this does not happen, precluding $12a_{1206}$ from being $3$-periodic:
\vskip2mm
\begin{center}
\begin{tabular}{|c|c||c|c|c|c|c|c|c|c|c|c|} \hline 
\multirow{2}{*}{$12a_{1206}$}&$d(Y_4,s)$  & $-\frac{9}{7}$  & $-1$  &  $-\frac{5}{7}$ & $-\frac{3}{7}$ & $-\frac{1}{7}$ & $\frac{1}{7}$  &$\frac{3}{7}$  & $\frac{5}{7}$ & $1$ & $\frac{9}{7}$ \cr \cline{2-12} 
&Multiplicity of $d(Y_4,s)$ & $2$ &  $4$ & $4$ & $6$ & $6$ & $6$ & $6$ & $4$ & $8$ & $2$ \cr \hline 
\end{tabular}
\end{center}

Somewhat similar to the case of $12a_{376}$ above, the knots in row two of \eqref{SpecialKnotsNeedingExtraTreatment} have two-fold covers $Y$ with $H_1(Y;\mathbb Z) \cong \mathbb Z_3^k$ with $k=4,5$. However, their corresponding hypothesized quotient knots $\overline K$ have two-fold covers $\overline Y$ with $|H_1(\overline Y;\mathbb Z)| = 3,9$, and thus by Corollary \ref{CorollaryForTheCaseOfQPrimaryGroups}, the action of $F_*$ on $H_1(Y;\mathbb Z)$ cannot be the identity. It thus must equal $F_* = 28 \cdot \text{id}$ or $F_*=55 \cdot \text{id}$ (if $k=4$) and $F_*=82\cdot \text{id}$ or $F_*=162\cdot \text{id}$ (if $k=5$). The fixed point set of $F_*$ has cardinality $27$ (if $k=4$) or $81$ (if $k=5$), showing that $81-27 = 54$ (if $k=4$) or $243-81 = 162$ (if $k=5$) of the correction terms of $Y$ must come with values that have multiplicity a multiple of $3$. An explicit computation shows that this is not the case for any of these knots. For example, for $12a_{425}$, we have $H_1(Y;\mathbb Z) = {\mathbb Z}_{3^4},\, H_1(\overline Y;\mathbb Z) =0,$ and the correction terms  and their multiplicities are given by 
\vskip2mm
\begin{center}
\begin{tabular}{|c|c||c|c|c|c|c|} \hline 
\multirow{2}{*}{$12a_{425}$}&$d(Y,s)$  & $-\frac{8}{9}$  & $\frac{4}{9}$  &  $-\frac{2}{9}$ & $0$ & All others  \cr \cline{2-7} 
&Multiplicity of $d(Y,s)$ & $4$ &  $4$ & $6$ & $9$ & $2$\cr \hline 
\end{tabular}
\end{center}
\vskip2mm

Finally, the two knots in the last row, if $3$-periodic, are forced by the Murasugi Conditions to have  Alexander polynomials for quotient knots equal to $4-7t+4t^{2}$
 (in the case of $12a_{634}$) and $1-t+t^{2}$ (in the case of $12a_{780}$).
 The knot $K=12a_{634}$ has $2$-fold cyclic branched cover $Y$ with $H_1(Y;\mathbb Z) \cong \mathbb Z_{3}\oplus \mathbb Z_{9}\oplus  \mathbb Z_5$ and quotient knot $\overline K$ with $2$-fold cyclic branched cover $\overline Y$ with $H_1(\overline Y;\mathbb Z) \cong \mathbb Z_3\oplus \mathbb Z_5$. 
Thus the only sensible choice of $\ell$ in Theorem \ref{FirstEquivarianceTheorem} is that of $\ell = 5$, rendering the said theorem ineffective as $H_1(Y;\mathbb Z)_5 \cong H_1(\overline Y;\mathbb Z)_5$. On the other hand, condition \eqref{RelationForTransferMap} (as in the proof of 
Theorem \ref{AbsenceOfL}) can be satisfied for certain choices of maps $\Pi_*$ and $\mu_*$, not allowing us to preclude the equality $F_*=$id. Accordingly, the correction term methods cannot be brought to bear on this knot. 

The knot $K=12a_{780}$ has two-fold cover $Y$ with $H_1(Y;\mathbb Z)\cong \mathbb Z_5\oplus \mathbb Z_{5} \oplus \mathbb Z_{9}$ and quotient knot $\overline K$, if it exists, with two-fold cover $\overline Y$ with $H_1(\overline Y;\mathbb Z) \cong \mathbb Z_3$. Thus Theorem \ref{FirstEquivarianceTheorem} with $\ell=5$ applies here, but alas all the correction terms corresponding to $\mathbb Z_5\oplus \mathbb Z_5-\{0\} = H_1(Y;\mathbb Z)_5-\{0\}$ do come with values that all have multiplicity $3$. 
Theorem \ref{AbsenceOfL} does not apply to the 3-torsion here. This makes $K=12a_{780}$ the only knot among twelve crossing alternating knots that passes the Murasugi Conditions, the Homology Condition and the correction terms condition, but the stronger factorization conditions over cyclotomic integers obtained by Davis-Livingston in \cite{D--L} show that this knot is not $3$-periodic.  (The knot $12a_{634}$ satisfies Davis-Livingston conditions.)

Lastly, there are no 12-crossing alternating knots that pass the homological $q$-periodicity obstructions for $q=5,7,11$. 
\subsection{Thirteen crossing alternating knots}
The largest genus of any 13-crossing alternating knot is 6 (realized only by the $(13,2)$-torus knot $13a_{4878}$). Edmonds' Conditions then show that the only possible odd prime periods $q$ of this family of knots, are $q=3,5,7,11,13$. 

There are 29 alternating 13-crossing knots that pass the homological 3-periodicity obstructions, but only three that pass it for period 5, namely
$$13a_{2142}, \, 13a_{2907}, \, 13a_{3010}.$$ 
All three of these are excluded from actually being 5-periodic by the Heegaard Floer correction terms. 

The knot $K_1=13a_{242}$ has Alexander polynomial $\Delta_{K_1}(t) = (1-t+t^2)^5$, and passes that Murasugi condition with hypothetical quotient knot $\overline K_1$ with $\Delta_{\overline K_1}(t) = 1-t+t^2$ and $\lambda =1$. An easy computation finds that  $H_1(Y_1;\mathbb Z) \cong \mathbb Z_9\oplus \mathbb Z_{27}$ and $H_1(\overline Y_1;\mathbb Z)\cong \mathbb Z_3$. The only way the Homology Condition is then satisfied is if $H_1(Y_1;\mathbb Z)_3/H_1(\overline Y_1;\mathbb Z)_3\cong \mathbb Z_9\oplus \mathbb Z_9$, in which case there need to be 80 correction terms that come with multiplicites divisible by 5. This is not the case as the next table shows:
\vskip2mm
\begin{center}
\begin{tabular}{|c|c||c|c|c|c|c|c|c|} \hline 
\multirow{2}{*}{$13a_{2142}$}&$d(Y_1,s)$  & $-\frac{29}{18}$  & $-\frac{17}{18}$  &  $-\frac{1}{2}$ & $-\frac{5}{18}$ & $\frac{1}{6}$ & $\frac{7}{18}$  &$\frac{19}{18}$    \cr \cline{2-9} 
&Multiplicity of $d(Y_1,s)$ & $6$ &  $14$ & $9$ & $18$ & $18$ & $12$ & $4$  \cr \hline 
\end{tabular}
\end{center}

The knot $K_2=13a_{2907}$ has Alexander polynomial $\Delta_{K_2}(t) = (2t-3)(3t-2)(1-3t+t^2) \equiv (1+t)^4 \, (\text{mod } 5)$, which passes the Murasugi condition with hypothetical quotient knot $\overline K_2$ with $\Delta_{\overline K_2}(t)=1$ and $\lambda =2$. As $H_1(Y_2;\mathbb Z) \cong \mathbb Z_{125}$ and $H_1(\overline Y_2;\mathbb Z)\cong 0$, Corollary \ref{CorollaryForTheCaseOfQPrimaryGroups}  then ensures that $|Fix(F_*)|\le 5$, forcing at least 120 correction terms of $Y_2$ to come with multiplicities divisible by 5. While $Y_2$ has 53 distinct values for its correction terms, most of them come with multiplicity 2, with a small number of exceptions:
\vskip2mm
\begin{center}
\begin{tabular}{|c|c||c|c|c|c|} \hline 
\multirow{2}{*}{$13a_{2907}$ and $13a_{3010}$}&$d(Y_2,s)$  & $-\frac{2}{5}$  & $0$  &  $\frac{2}{5}$ & All others    \cr \cline{2-6} 
&Multiplicity of $d(Y_2,s)$ & $10$ &  $5$ & $10$ & $2$  \cr \hline 
\end{tabular}
\end{center}

The knot $K_3=13a_{3010}$ has Alexander polynomial $\Delta_{K_3}(t) = (1-3t+t^2)(1-t+t^-t^3+t^4)^2 \equiv (1+t)^{10} \, (\text{mod } 5)$, and passes the Murasugi condition with hypothetical quotient knot $\overline K_3$ with $\Delta _{\overline K_3}(t) =1-3t+t^2$ and $\lambda =1$. It is easy to compute $H_1(Y_3;\mathbb Z) \cong \mathbb Z_{125}$ and $H_1(\overline Y_3;\mathbb Z)\cong \mathbb Z_5$, and so Corolllary \ref{CorollaryForTheCaseOfQPrimaryGroups}  implies that $|Fix(F_*)|\le 25$ showing that at least 100 correction terms must come with multiplicities divisible by 5. Indeed, the correction terms of $Y_3$ fit the same table of multiplicities as the correction terms for $Y_2$, and so $13a_{3010}$ is not 5-periodic.  We note that while the correction terms for $Y_2$ and $Y_3$ follow identical multiplicity patterns, the correction terms for $Y_2$ do not equal those for $Y_3$. For instance, $-\frac{212}{125}$ is the value of a correction term for $Y_2$ that does not occur among the correction terms for $Y_3$.  

No 13-crossing alternating knots passes the homological $q$-periodicity obstructions for $q=7,11$, while the only knot that passes the homological 13-periodicity obstructions is $13a_{4878}$, the $(13,2)$-torus knot. 
\subsection{Fourteen crossing alternating knots}
The maximum genus among 14-crossing alternating knots is $6$, showing that the only possible odd prime periods that may occur for these knots are $q=3,5,7,11,13$. 

There are 49 knots that pass the homological 3-periodicity obstructions, but only 2 knots pass the homological 5-periodicity obstructions, namely  
$$ 14a_{11685},\, 14a_{14294}. $$
Both of these knots are excluded from having period 5 by the Heegaard Floer correction terms. Indeed, the knot $K_1=14a_{11685}$ has Alexander polynomial $\Delta _{K_1}(t) = (2t-3)(3t-2)(1-3t+t^2) \equiv (1+t)^4 \,(\text{mod }5)$, passing the Murasugi condition with possible quotient knot $\overline K_1$ with Alexander polynomial $\Delta_{\overline K_1}(t) = 1$ and $\lambda =2$. One easily obtains $H_1(Y_1;\mathbb Z)\cong \mathbb Z_{125}$ and $H_1(\overline Y_1;\mathbb Z)\cong 0$, so that Corollary \ref{CorollaryForTheCaseOfQPrimaryGroups} implies that $Y_1$ has at least 120 correction terms with multiplicity divisible by 5. However, this does not occur.
\vskip2mm
\begin{center}
\begin{tabular}{|c|c||c|c|c|c|c|} \hline 
\multirow{2}{*}{$14a_{11685}$}&$d(Y_1,s)$  & $-\frac{8}{5}$  & $-\frac{2}{5}$  &  $0$ & $\frac{2}{5}$ & All others    \cr \cline{2-7} 
&Multiplicity of $d(Y_1,s)$ & $4$ &  $10$ & $3$ & $6$ & $2$  \cr \hline 
\end{tabular}
\end{center}

The knot $K_2=14a_{14294}$ has Alexander polynomial $\Delta_{K_2}(t) = (1-3t+t^2)(10-31t+43t^2-31t^3+10t^4) \stackrel{.}{\equiv} (1+t)^4 \, (\text{mod } 5)$ and so, as with $K_1$, it passes Murasugi's condition with possible quotient knot $\overline K_2$ with $\Delta _{\overline K_2}(t) = 1$ and $\lambda =2$. Since  $H_1(Y_2;\mathbb Z)\cong \mathbb Z_5\oplus \mathbb Z_{125}$ and $H_1(\overline Y_2;\mathbb Z)\cong 0$, Theorem \ref{AbsenceOfL} guarantees that there are at least $600$ correction terms with multiplicities divisible by 5. This knot comes close to having such correction terms. Indeed, each of the values $\frac{\lambda }{125}$ with 
\begin{align*}
\lambda & \in \{ -99, -91, -89, -79, -71, -69, -61, -59, -51, -49, -41, -39, -31, -29,  \cr
& \quad \quad \quad \quad \quad  -21, -19, -11, -9, -1, 1, 9, 11, 19, 21, 29, 31, 41, 49, 51, 59, 81, 89  \},
\end{align*}
represents the value of a correction term of $Y_2$ with multiplicity 10 (for a total of 320 correction terms).  Each of $-\frac{3}{5}$, $\pm \frac{1}{5}$ is the value of a correction term of $Y_2$ of multiplicity 30 (adding another 90 correction terms).  Moreover, $-\frac{7}{5}$ has multiplicity 12 and $\frac{3}{5}$ has multiplicity 18, each possibly contributing correction terms with multiplicity 10 (for an additional 20 correction terms). All other correction terms have multiplicities 2, 3, 4, 6 and 8. Overall, there are 430, out of 600 needed, correction terms with multiplicities divisible by 5. Thus $14a_{14294}$ is not $5$-periodic. 

There are 7 alternating 14-crossing knots that pass the homological 7-periodicity obstructions: 
\begin{equation} \label{14crossing7periodic}
14a_{3023}, \, 14a_{6681}, \,  14a_{7378}, \,  14a_{12332}, \, 14a_{12702}, \, 14a_{15044}, \,   \boxed{14a_{19470}}. 
\end{equation}
Of these, the knot $14a_{19470}$ is 7-periodic (cf. Figure \ref{knot14a19470}) while the remaining knots are excluded from being 7-periodic by the Heegaard Floer correction terms. Letting $L_1, \dots, L_6$ denote the six non-framed knots in \eqref{14crossing7periodic}, we find that $\Delta_{L_i}(t) = (2-3t+2t^2)^3 \equiv (1+t)^6 \, (\text{mod } 7)$ for each $i=1,\dots, 6$, forcing the hypothetical quotient knot $\overline L_i$ to have trivial Alexander polynomial (and $\lambda _i=2$) in order to pass the Murasugi condition. One finds that $H_1(X_i;\mathbb Z)\cong \mathbb Z_{343}$ for each $i=1,\dots, 6$, and so  Corollary \ref{CorollaryForTheCaseOfQPrimaryGroups} implies the existence of at least 336 correction terms with multiplicity divisible by 7, a condition that is violated by each of the six knots. 

No 14-crossing alternating knot passes the homological $q$-periodicity obstructions for $q=11,13$.
\subsection{Fifteen crossing alternating knots}
The largest genus of any 15-crossing alternating knot is 7 (attained uniquely by the $(15,2)$-torus knot $15a_{85263}$), and accordingly the only possible odd prime periods for these knots are $q=3,5,7,11,13$. 

There are 133 knots that pass the homological 3-periodicity obstructions. The only 11 knots that may have period 5 are 
$$
\begin{array}{c}
15a_{23599},\, 15a_{23902},\, 15a_{40549},\,  15a_{53966}, \, \boxed{15a_{64035}}, \, 15a_{69121}, \,
15a_{76651},\, 15a_{80526},\cr \boxed{15a_{84903}},\, \boxed{15a_{85262}},\, \boxed{15a_{85263}}. 
\end{array}
$$
The four boxed knots are 5-periodic. Indeed, $15a_{85263}$ is the $(15,2)$-torus knot, the knot $15a_{85262}$ is the 5-stranded pretzel knot $P(-3,-3,-3,-3,-3)$, and 5-periodic diagrams for $15a_{64035}$ and $15a_{84903}$ are given in Figure \ref{15crossingPeriodicKnots}. Each of the remaining 7 knots is excluded from being 5-periodic by the Heegaard Floer correction terms. 

Let $K_1$ be either of the knots $15a_{23599}$ or $15a_{23902}$, then $\Delta _{K_1}(t) = (1-3t+t^2)^3(1-t+t^2-t^3+t^4) \equiv (1+t)^{10} \, (\text{mod } 5)$. This Alexander polynomial passes the Murasugi test with quotient knot $\overline K_1$ with $\Delta _{\overline K_1}(t) =1-3t+t^2$ and $\lambda = 1$. We find $H_1(Y_1;\mathbb Z) \cong \mathbb Z_5\oplus \mathbb Z_{125}$ and $H_1(\overline Y_1;\mathbb Z)\cong \mathbb Z_5$. Theorem \ref{AbsenceOfL} then implies the existence of at least 600 correction terms with multiplicities divisible by 5. However, $15a_{23599}$ only has 340 such correction terms, and $15a_{23902}$ only 320.  

Let $K_2=15a_{40549}$, then $\Delta _{K_2}(t) = (1-2t+4t^2-3t^3+t^4)(1-3t+4t^2-2t^3+t^4)(1-t+t^2-t^3+t^4) \equiv (1+t+t^2+t^3)^4 \, (\text{mod } 5)$, which passes the Murasugi condition with $\Delta _{\overline K_2}(t) = 1$ and $\lambda =4$. As $H_1(Y_2;\mathbb Z) \cong \mathbb Z_5\oplus \mathbb Z_{11}\oplus \mathbb Z_{11}$ and $H_1(\overline Y_2;\mathbb Z)\cong 0$,  Theorem \ref{ObstructionToQPeriodicityWhenFixedPointSetIsTrivial} with $\ell =11$ shows that each correction term $d(Y_2,s)$, $s\in H_1(Y_2;\mathbb Z)_{11}-\{0\}$ has to have a multiplicity divisibly by 5 (for a total of 120 such correction terms). However, only 70 such correction terms exist, precluding 5-periodicity of $15a_{40549}$. 

Let $K_3$ be any of the knots $15a_{53966}$, $15a_{69121}$, or $15a_{80526}$, then $\Delta_{K_3}(t) = (2-3t+2t^2)(1-3t+3t^2-3t^3+t^4)^2 \equiv [2(1+t+t^2)]^5 \, (\text{mod } 5)$. This polynomial passes the Murasugi condition with $\Delta_{\overline K_3}(t) = 2-3t+2t^2$ and $\lambda = 3$. The first homology groups of $Y_3$ and $\overline Y_3$ are $H_1(Y_3;\mathbb Z) \cong \mathbb Z_{7}\oplus \mathbb Z_{11} \oplus \mathbb Z_{11}$ and $H_1(\overline Y_3;\mathbb Z)\cong \mathbb Z_7$. As with the previous example, Theorem \ref{ObstructionToQPeriodicityWhenFixedPointSetIsTrivial} with $\ell = 11$ guarantees the existence of 120 correction terms $d(Y_3,s)$, $s\in H_1(Y_3;\mathbb Z)_{11}-\{0\}$, each with a multiplicity that is divisible by 5. This condition is not realized for any of the three possible knots in this paragraph. 

Finally, consider $K_4 = 15a_{76651}$, then $\Delta _{K_4}(t) = 10-61t+141t^2-181t^3+141t^4-61t^5+10t^6 \equiv (1+t)^4 \, (\text{mod } 5 )$ which passes Murasugi's condition with $\Delta _{\overline K_4}(t) = 1$ and $\lambda = 2$. The first homologies of $Y_4$ and $\overline Y_4$ are $H_1(Y_4;\mathbb Z)\cong \mathbb Z_5 \oplus \mathbb Z_{11} \oplus \mathbb Z_{11}$ and $H_1(\overline Y_4;\mathbb Z)\cong 0$. Theorem \ref{ObstructionToQPeriodicityWhenFixedPointSetIsTrivial} with $\ell = 11$ provides for the existence of 120 correction terms of $Y_4$ with multiplicity divisibly by 5. However, there are only 90 such correction terms. 
\vskip2mm

Only a single 15-crossing alternating knot passes the homological 7-periodicity obstruction, namely 
$$ 15a_{23046}.$$
This knot $L$ has Alexander polynomial $\Delta_L(t) = (2-3t+2t^2)^3 \equiv (1+t)^6\, (\text{mod } 7)$, passing the Murasugi condition with $\Delta _{\overline L}(t) = 1$ and $\lambda =2$. Its 2-fold branched cover $X$ has first homology $H_1(X;\mathbb Z)\cong \mathbb Z_{343}$, and thus Corollary \ref{CorollaryForTheCaseOfQPrimaryGroups} implies the existence of at least 336 correction terms of $X$ with multiplicities divisible by 7. A direct calculation shows that there are only 14 such correction terms. 
\vskip2mm

Finally, no 15-crossing alternating knot passes the homological $q$-periodicity obstruction for $q=11, 13$, completing the proof of Theorem \ref{131415}. 

\bigskip
\footnotesize
\noindent\textit{Acknowledgments.}
We thank Chuck Livingston for sharing his Maple code for computing twisted Alexander polynomials. We would like to thank the two anonymous Referees for providing useful comments and suggestions, and specifically for suggesting Question \ref{QuestionAboutHyperbolicStructure}. 

S. Jabuka was partially supported by grant \#246123 from the Simons Foundation. S.~Naik was partially supported by grant number DMS-0709625 from the National Science Foundation.

\end{document}